\documentclass[11pt]{amsart}

\usepackage{algorithm}
\usepackage{algorithmic}
\usepackage{amscd}
\usepackage{amsfonts}
\usepackage{amsmath}
\usepackage{amssymb}
\usepackage{amsthm}
\usepackage{color}
\usepackage{enumerate}
\usepackage{epsfig}
\usepackage{graphicx}
\usepackage{latexsym}
\usepackage{mathrsfs}
\usepackage{pxfonts}
\usepackage{subfig}
\usepackage[english]{babel}
\usepackage[left=2cm,right=2cm,top=2cm,bottom=2cm]{geometry}
\usepackage{lscape}
\usepackage{algorithm}
\usepackage{algorithmic}

\theoremstyle{definition}
\newtheorem{Definition}{Definition}[section]
\newtheorem{Example}{Example}[section]
\newtheorem{Remark}{Remark}[section]

\theoremstyle{plain}
\newtheorem{Proposition}[Definition]{Proposition}
\newtheorem{Lemma}[Definition]{Lemma}
\newtheorem{Theorem}[Definition]{Theorem}

\newtheorem{Assumption}[Definition]{Assumption}

\numberwithin{equation}{section}
\allowdisplaybreaks[4]

\makeatletter
\newcommand{\biggg}[1]{{\hbox{$\left#1\vbox to 20.5pt{}\right.\n@space$}}}
\newcommand{\Biggg}[1]{{\hbox{$\left#1\vbox to 23.5pt{}\right.\n@space$}}}
\newcommand{\bigggg}[1]{{\hbox{$\left#1\vbox to 26.5pt{}\right.\n@space$}}}
\newcommand{\Bigggg}[1]{{\hbox{$\left#1\vbox to 29.5pt{}\right.\n@space$}}}
\newcommand{\biggggg}[1]{{\hbox{$\left#1\vbox to 32.5pt{}\right.\n@space$}}}
\newcommand{\Biggggg}[1]{{\hbox{$\left#1\vbox to 35.5pt{}\right.\n@space$}}}
\newcommand{\bigggggg}[1]{{\hbox{$\left#1\vbox to 38.5pt{}\right.\n@space$}}}
\newcommand{\Bigggggg}[1]{{\hbox{$\left#1\vbox to 41.5pt{}\right.\n@space$}}}
\makeatother

\newcommand{\Bigggl}{\mathopen\Biggg}

\newcommand{\Bigggr}{\mathclose\Biggg}

\begin{document}
\title[Simulations of an Euler-Maruyama type approximation of very high dimensional SDEs]{An efficient weak Euler-Maruyama type approximation scheme of very high dimensional SDEs by orthogonal random variables}
\author{Jir\^o Akahori}
\address{J. Akahori. Department of Mathematical Sciences, Ritsumeikan University, 1-1-1, Nojihigashi, Kusatsu, Shiga, 525-8577, Japan}
\curraddr{}
\email{akahori@se.ritsumei.ac.jp}
\thanks{}
\author{Masahiro Kinuya}
\address{M. Kinuya. Okasan Securities Co., Ltd., 1-17-6, Nihonbashi, Chuo-ku, Tokyo, 103-8278, Japan}
\curraddr{}
\email{}
\thanks{}
\author{Takashi Sawai}
\address{T. Sawai. Department of Mathematical Sciences, Ritsumeikan University, 1-1-1, Nojihigashi, Kusatsu, Shiga, 525-8577, Japan}
\curraddr{}
\email{}
\thanks{}
\author{Tomooki Yuasa}
\address{T. Yuasa. Department of Mathematical Sciences, Ritsumeikan University, 1-1-1, Nojihigashi, Kusatsu, Shiga, 525-8577, Japan}
\curraddr{}
\email{to-yuasa@fc.ritsumei.ac.jp}
\thanks{}
\date{}

\begin{abstract}
\label{sec:abst}
We will introduce Euler-Maruyama approximations given by an orthogonal system in $L^{2}[0,1]$ for high dimensional SDEs, which could be finite dimensional approximations of SPDEs.
In general, the higher the dimension is, the more one needs to generate uniform random numbers at every time step in numerical simulation.
The schemes proposed in this paper, in contrast, can deal with this problem by generating very few uniform random numbers at every time step.
The schemes save time in the simulation of very high dimensional SDEs.
In particular, we conclude that an Euler-Maruyama approximation based on the Walsh system is efficient in high dimensions.
\end{abstract}
\maketitle

\section{Introduction}
\label{sec:1}
\subsection{Simulation of a high dimensional SDE}
\label{sec:1.1}
Let $X$ be a unique solution of the $d$-dimensional (time-homogeneous Markovian type) stochastic differential equation (SDE for short)
\begin{align}
\label{eq:1.1}
{\rm d}X_{t}=\sigma(X_{t}){\rm d}W_{t}+b(X_{t}){\rm d}t, \quad t \geq 0, \quad X_{0} \equiv x_{0} \in {\mathbb R}^{d}.
\end{align}
Here, $W$ is a $d$-dimensional Brownian motion starting at the origin, and the coefficients $\sigma: {\mathbb R}^{d} \to {\mathbb R}^{d} \otimes {\mathbb R}^{d}$ and $b: {\mathbb R}^{d} \to {\mathbb R}^{d}$ are sufficiently regular.
Our purpose in this paper is to provide efficient weak approximations for the quantity ${\mathbb E}[f(X_{T})]$ in high dimensions for a given function $f$ and a time horizon $T>0$.
The quantity can mean, for example, a fair price of European type derivatives in a financial market, where $T$ is the maturity and $f: {\mathbb R}^{d} \to {\mathbb R}$ is a pay-off function.
In financial practice, its numerical value is of great importance.
With an explicit finite dimensional expression, the problem reduces to a standard numerical analysis such as approximation of a finite --- desirably less than three --- dimensional integral using a Riemann sum, but except for some simple cases such an expression is not available.
A simple but most frequently used way to reduce it to a finite dimensional integration is the so-called Euler-Maruyama approximation (EM scheme for short), which is typically given by:
\begin{align}
\label{eq:1.2}
X_{0}^{(n)}=x_{0}, \quad X_{\ell T/n}^{(n)}=X_{(\ell-1)T/n}^{(n)}+\sigma(X_{(\ell-1)T/n}^{(n)})(W_{\ell T/n}-W_{(\ell-1)T/n})+b(X_{(\ell-1)T/n}^{(n)})\frac{T}{n}, \quad \ell \in \{1,2,\ldots,n\}.
\end{align}
The numerical integration is calculated, or should we say, simulated, by the Monte Carlo method.

It is widely-recognized that in the standard EM scheme (see e.g. \cite{KEPPE,MNG}), apart from the numerical integration error, the finite dimensional reduction error is of order $n^{-1/2}$ in the strong sense, and of order $n^{-1}$ in the weak sense, under some regularity conditions on $\sigma$ and $b$.
Moreover, even if the system $\{W_{\ell T/n }-W_{(\ell-1)T/n}\}_{\ell=1}^{n}$ (called the Gaussian system hereinafter) is replaced with random variables ``simulating Brownian increments", by which we mean random variables sharing moments up to some order with the increments of the standard Brownian motion, the weak error is still of order $n^{-1}$ (see e.g. \cite{KEPPE}).
This implies that we can change and select the system in accordance with various requirements.

In the present paper, we will propose new weak EM schemes that work {\it faster} than the existing schemes in \underline{very high dimensions};
when $ d $ is very large.
In our schemes, to simulate high dimensional Brownian increments we use a \underline{single} uniform random variable taking values in the set of integers.
Our schemes are theoretically shown to have the same convergence order $n^{-1}$ in the weak sense as the standard EM scheme, and several numerical experiments in high dimensions show that the accuracy of our schemes are comparable to the standard EM scheme and the computation time of our schemes is much faster than the standard EM scheme.

Recent year witnessed the rapidly growing demands and supplies for efficient numerical simulation schemes for high-dimensional stochastic differential equations.
The demands are from 
\begin{itemize}
    \item Mean-field (or McKean-Vlasov)
    type SDE, and other multi-scale models,
    \item (Galerkin) approximation of Stochastic Partial Differential Equation (SPDE for short),
    \item Approximation of infinite particle systems in relations with random matrices.
    \item and so on.
\end{itemize}
The supplies are mainly due to the explosive development of computing ability of computers, but the recent success of machine-learning/deep learning technologies 
should not be dismissed.
An almost dimension-free scheme for solving forward-backward stochastic differential equations and related non-linear partial differential equations using deep learning
has already gained huge popularity within these couple of years (\cite{EWHJJA}, \cite{HJJAEW},  \cite{HCPHWX}).
Also, many numerical schemes
that are claimed to be efficient
for each of the above mentioned problems
have been proposed.
For numerical analysis of the mean field SDE,
it dates back, to the best of the author's knowledge, to the paper by S. Ogawa \cite{OgawaI}, and some  basic results are exploited in
\cite{OgawaII}, \cite{OgawaIII},  \cite{Kohatsu-Ogawa}, and also \cite{An-Ko}. 
The mean-field game (see e.g. \cite{LL}) then became a driving force
of further studies. 
For the Galerkin approximation of SPDEs, see e.g. \cite{JAKEP1} and \cite{JAKEP2}, 
and for EM approximations of infinite particle systems in relations with 
random matrices, 
see e.g. \cite{Taguchi} and references therein. 

Contrasting with such specific schemes,
ours is universally applicable
since we are just looking at the simulation of
Brownian increments.

\subsection{Contributions of the present paper}
In this paper, we provide two systems based on an orthogonal system of $L^{2}[0,1]$ in order to deal with SDEs in ``very high dimensions", by which we mean, let's say, $d \sim 2^{32}$.
In such high dimensions, generating a Gaussian system is heavily time consuming since we need to generate many uniform random numbers at every time step.
In contrast, our schemes use very few uniform random numbers (If using uniform $32$-bit random numbers and $d \leq 2^{31}$, our schemes only use a {\bf single} uniform random number at every time step).

It is true that, mathematically speaking, any Gaussian system can also be generated from a single uniform random variable if we were given an ideally uniform one from $[0,1]$, which is equivalent to infinitely many binary distributed random numbers.
In reality, however, a {\em uniform} random number is actually a finite sequence of binary random numbers, which is, by the dyadic expansion, equivalent to {\it uniform} random integers in a finite set. The {\em Mersenne twister}\footnote{Many variants have been proposed, improving the original one, mainly by M. Matsumoto and his collaborators. See e.g. \cite{Harase} and references therein.}, the most reliable pseudo random number generating algorithm offered by Matsumoto-Nishimura \cite{MMNT}, generates uniform $32$-bit numbers, or, equivalently, {\it uniform} integers over $\{1,2, \ldots, 2^{32}\}$.

The very heart of our schemes is an algorithm to generate ``simulated Brownian increments" in very high dimensions out of a given set of uniform integers.
We will present two distinct schemes; one is based on the Haar system of orthogonal functions in $ L^2 [0,1] $, and the other 
comes from the Walsh system. 
In respect of computational time, 
both are experimentally shown to be
more efficient than the Gaussian 
schemes (see Section \ref{sec:4}). 
Between the two of our new schemes, 
the former is a bit more efficient, 
but with an on-line algorithm 
implied by Theorem \ref{theo:2.4}, 
the difference can be reduced (see Section \ref{sec:4}). 
In respect of accuracy, however, the former may suffer 
from problems caused by greater higher moments. 
One of the problems is illustrated by the difference of the bounds
\eqref{Haarbound} and \eqref{Walshbound} in Theorem \ref{theo:3.3}. 
The difference is more clearly seen in the simple example of Proposition \ref{prop:3.new}

\subsection{Notations}
\label{sec:1.3}
Throughout this paper, we use $d, n, m \in {\mathbb N}$ as the dimension of SDEs, the number of partitions of the closed interval $[0,T]$ (the number of time steps on the closed interval $[0,T]$ for the EM scheme) and the number of
 Monte Carlo trials, respectively.
$\delta_{i,j}$, $i,j \in {\mathbb N}$ denotes the Kronecker delta, i.e., if $i=j$, then $\delta_{i,j}=1$, otherwise, if $i \neq j$, then $\delta_{i,j}=0$.
The components of a vector are denoted by superscripts without parentheses.
The row vectors of a matrix are denoted by superscripts without parentheses and the column vectors of a matrix are denoted by subscripts.
On the other hand, superscripts with parentheses and subscripts mean the dependence on parameters.

The Euclidean norm on ${\mathbb R}^{d}$ and ${\mathbb R}^{d} \otimes {\mathbb R}^{d}$ are denoted by $|x|:=(\sum_{i=1}^{d}|x^{i}|^{2})^{1/2}$, $x \in {\mathbb R}^{d}$ and $|x|:=(\sum_{i,j=1}^{d}|x_{j}^{i}|^{2})^{1/2}$, $x \in {\mathbb R}^{d} \otimes {\mathbb R}^{d}$, respectively.

Let $h \in \{1,2,3,4\}$.
The space of real valued polynomial growth functions on ${\mathbb R}^{d}$ with polynomially bounded continuous derivatives up to $h$ is denoted by $C_{P}^{h}({\mathbb R}^{d})$, that is
$$
C_{P}^{h}({\mathbb R}^{d}):=\left\{f \in C^{h}({\mathbb R}^{d}) \,;\, 
\begin{array}{ll}
\text{There exist } r \in {\mathbb N} \cup \{0\} \text{ and a positive constant } C \text{ such that } \\
\displaystyle{\text{for any } y \in {\mathbb R}^{d}, \max_{\substack{i_{1},i_{2},\ldots,i_{k} \in \{1,2,\ldots,d\} \\ k \in \{1,2,\ldots,h\}}}\left|\frac{\partial^{k}f(y)}{\partial y^{i_{1}}\partial y^{i_{2}}\cdots\partial y^{i_{k}}}\right| \leq C\left(1+|y|^{2r}\right)}.
\end{array}
\right\}.
$$

\subsection{Outline}
\label{sec:1.4}
This article is divided as follows:
In Section \ref{sec:2}, we will introduce two Euler-Maruyama approximations based on the Haar system and the Walsh system for high dimensional SDEs.
Moreover, we will show that the 1st, 2nd and 3rd moments of these systems are the same as of the Gaussian system.
In Section \ref{sec:3}, we will state the error estimate for the weak convergence of the EM scheme in the general system.
This will imply that the EM schemes by the Haar system and the Walsh system have the same weak order $1$ of convergence as the standard EM scheme by the Gaussian system.
As the same time, our estimate 
suggests that 
the error 
grows very rapidly as the dimension gets higher  
in the Haar case, while 
the scheme with the Walsh system 
behaves far more nicely.
In Section \ref{sec:4}, to confirm that our schemes are more efficient than the standard EM scheme, we will perform some numerical experiments.
In Section \ref{sec:5}
(Appendix), we will prove the error estimate stated in Section \ref{sec:3} using the It\^o Taylor expansion (the Wagner-Platen expansion).

\section{Euler-Maruyama approximation for High Dimensional SDEs}
\label{sec:2}
\subsection{Simulation by mimicking random variables}
\label{sec:2.1}
In this section, we introduce two efficient algorithms for Euler-Maruyama approximations of a high dimensional SDE \eqref{eq:1.1} by orthogonal random variables after introducing the framework for the schemes to work on.

Let $n \in {\mathbb N}$ be the number of partitions of the closed interval $[0,T]$ and $(\Delta Z_{\ell}^{(n)})_{\ell=1}^{n}$ be $d$-dimensional i.i.d. random variables.
We consider the following Euler-Maruyama approximation of the equation \eqref{eq:1.1} given by $(\Delta Z_{\ell}^{(n)})_{\ell=1}^{n}$:
\begin{align}
\label{eq:2.1}
X_{0}^{(n)}=x_{0}, \quad X_{\ell T/n}^{(n)}=X_{(\ell-1)T/n}^{(n)}+\sigma(X_{(\ell-1)T/n}^{(n)})\Delta Z_{\ell}^{(n)}+b(X_{(\ell-1)T/n}^{(n)})\frac{T}{n}, \quad \ell \in \{1,2,\ldots,n\}.
\end{align}
Here, $0<T/n<2T/n<\cdots<nT/n=T$ form equal time steps on $[0,T]$. 
We assume that the system $(\Delta Z_{\ell}^{(n)})_{\ell=1}^{n}$ satisfies, for any $\ell \in \{1,2,\ldots,n\}$,
\begin{align}
\label{eq:2.2}
{\mathbb E}\left[(\Delta Z_{\ell}^{(n)})^{j_{1}}\right]=0, \quad \forall j_{1} \in \{1,2, \ldots, d\},
\end{align}
\begin{align}
\label{eq:2.3}
{\mathbb E}\left[(\Delta Z_{\ell}^{(n)})^{j_{1}}(\Delta Z_{\ell}^{(n)})^{j_{2}}\right]=\frac{T}{n}\delta_{j_{1},j_{2}}, \quad \forall j_{1}, j_{2} \in \{1,2, \ldots, d\}
\end{align}
and 
\begin{align}
\label{eq:2.4}
{\mathbb E}\left[(\Delta Z_{\ell}^{(n)})^{j_{1}}(\Delta Z_{\ell}^{(n)})^{j_{2}}(\Delta Z_{\ell}^{(n)})^{j_{3}}\right]=0, \quad \forall j_{1}, j_{2}, j_{3} \in \{1,2, \ldots, d\}. 
\end{align}
Then the Euler-Maruyama approximation of \eqref{eq:1.1} given by $(\Delta Z_{\ell}^{(n)})_{\ell=1}^{n}$ has weak order $1$ of convergence even if $(\Delta Z_{\ell}^{(n)})_{\ell=1}^{n}$ is not the Gaussian system (the increments of a Brownian motion), which will be made more precise and proven as a corollary to a more general theorem. 
Note that without \eqref{eq:2.3} we can only prove that it has weak order $1/2$ of convergence in Theorem \ref{theo:3.3} below.

Our objective in the following subsections is to find systems $(\Delta Z_{\ell}^{(n)})_{\ell=1}^{n}$ having the following conditions.
\begin{itemize}
\item $(\Delta Z_{\ell}^{(n)})_{\ell=1}^{n}$ satisfies the moment conditions \eqref{eq:2.2}, \eqref{eq:2.3} and \eqref{eq:2.4}.
\item The Euler-Maruyama approximation given by $(\Delta Z_{\ell}^{(n)})_{\ell=1}^{n}$ is more efficient than the standard Euler-Maruyama approximation by the Gaussian system in high dimensions.
Here, the efficiency is measured by the balance between accuracy and computation time.
\end{itemize}

Our strategy is to generate many random numbers out of a single uniform random number in order to reduce the time consumed to generate random numbers in high dimensions.
Below we introduce two distinct constructions, one from Haar functions and the other from Walsh functions. 
As was already discussed, the generating random variable is practically uniform on the set of integers, or, equivalently, on a binary set $\{0,1\}^{K}$, for some $K \in {\mathbb N}$.
Mathematically speaking, Propositions \ref{prop:2.1}
and \ref{prop:2.2} are not really new. Here we state them to stress and confirm that 
the third moments vanish, which is  important specifically in our schemes. 
The odd-ordered map of Theorem \ref{theo:2.4} is new though the proof is elementary. It gives a practical algorithm 
which seems to endorse the Walsh scheme to as competitive as the Haar scheme. 
See Section \ref{sec:4} for experimental results
to indicate how the algorithm saves
the computational time.

To the best of our knowledge, our schemes are new. 
The cubature method proposed by T. Lyons and his collaborators (\cite{LV}, see also \cite{Crisan} for recent developments) might be a nearest in that 
they use discrete random variables that
are not fully independent 
to mimic high-dimensional Brownian increments. 
They are also saving computational time by attaining a better convergence rate.

\subsection{Mimicking by the Haar system}
\label{sec:2.2}
Let $K \in {\mathbb N} $ with $d \leq 2^{K-1}$ and
\begin{align*}
h_{k}^{(K)}(x)=
\left\{
\begin{array}{ll}
2^{(K-1)/2} & x=2k-1 \\
-2^{(K-1)/2} & x=2k \\
0 & \text{otherwise},
\end{array}
\right. \quad x \in \{1,2,\ldots, 2^K\}, \quad k \in \{1,2,\ldots,2^{K-1}\}.
\end{align*}
\begin{Proposition}
\label{prop:2.1}
Define
\begin{align}
\label{eq:2.5}
(\Delta Z^{(n)})^{j}:=h_{j}^{(K)}(U)\sqrt[]{\frac{T}{n}}, \quad j \in \{1,2,\ldots,d\},
\end{align}
where $U$ is a random variable distributed uniformly over $\{1,2,\ldots, 2^K\}$. 
Then the $d$-dimensional random variable $\Delta Z^{(n)}$ satisfies \eqref{eq:2.2}, \eqref{eq:2.3} and \eqref{eq:2.4}. 
\end{Proposition}
\begin{proof}
Let $p \in {\mathbb N}$ and $j_{1},j_{2},\ldots,j_{p} \in \{1,2,\ldots,d\}$.
We obtain
\begin{align}
\label{eq:2.6}
{\mathbb E}\left[\prod_{k=1}^{p}h_{j_{k}}^{(K)}(U)\right]
&=\sum_{x=1}^{2^{K}}\left(\prod_{k=1}^{p}h_{j_{k}}^{(K)}(x)\right)\mathbb{P}(U=x) \\ 
&=\sum_{x=1}^{2^{K}}\left(\prod_{k=1}^{p}\left({2^{(K-1)/2}\bf 1}_{\{2j_{k}-1\}}(x)-{2^{(K-1)/2}\bf 1}_{\{2j_{k}\}}(x)\right)\right)\frac{1}{2^{K}} \notag \\
&=
\left\{
\begin{array}{ll}
\displaystyle{2^{p(K-1)/2-K}\sum_{x=1}^{2^{K}}\left({\bf 1}_{\{2j_{1}-1\}}(x)-{\bf 1}_{\{2j_{1}\}}(x)\right)^{p}} & j_{1}=j_{2}=\cdots=j_{p} \\
0 & \text{otherwise}
\end{array}
\right. \notag \\
&=
\left\{
\begin{array}{ll}
\displaystyle{2^{(p-2)(K-1)/2}} & p \text{ is even and } j_{1}=j_{2}=\cdots=j_{p} \\
0 & \text{otherwise}.
\end{array}
\right. \notag
\end{align}
Equation \eqref{eq:2.6} implies \eqref{eq:2.2}, \eqref{eq:2.3} and \eqref{eq:2.4}.
\end{proof}

\subsection{Mimicking by the Walsh system}
\label{sec:2.3}
Let $K \in {\mathbb N}$ with $d \leq 2^{K-1}$.
We will denote by $\tau=(\tau_{1},\tau_{2},\ldots,\tau_{K})$ an element of the finite product set $\{-1,1\}^{K}$ of a two-point set $\{-1,1\} \subset {\mathbb R}$.
Endowed with the uniform distribution on $\{-1,1\}^{K}$, 
the coordinate maps $\tau \mapsto \tau_{i}$, $i \in \{1,2,\ldots,N\}$, regarded as random variables, 
are mutually independent and  identically distributed as
${\mathbb P}(\tau_{i}=\pm 1)=1/2$, $i \in \{1,2,\ldots,K\}$.

For non empty $ S \subset \{1,2, \cdots, K\} $, define 
\begin{equation*}
    \begin{split}
        \tau_S := \prod_{i \in S} \tau_i,
    \end{split}
\end{equation*}
and set $ \tau_\emptyset \equiv 1 $, that is, $ \tau _\emptyset (\omega) = 1 $
for all $ \omega \in \{-1,1\}^K $. 
Note that 
$ \tau_S $ for $ \emptyset \ne S \subset \{1, \cdots, K \} $  are all identically distributed. Further, if
$ S $ and $ S' $ are distinct,  
$\tau_{S}$ and $\tau_{S^{\prime}}$ are
not always independent but 
orthogonal (their covariance is $0$, i.e., their correlation is $0$).
Indeed, we obtain
$$
\tau_{S}\tau_{S^{\prime}}=\tau_{S \ominus S^{\prime}}\tau_{S \cap S^{\prime}}^{2}=\tau_{S \ominus S^{\prime}},
$$
where the symbol $\ominus $ is the symmetric difference of two sets. 
Then by the independence, we have
\begin{align}
\label{eq:2.7}
{\mathbb E}\left[\tau_{S}\tau_{S^{\prime}}\right]
=\prod_{j \in S \ominus S^{\prime}}{\mathbb E}\left[\tau_{j}\right]
=0.
\end{align}
Thus the system ${\mathscr W}_{K}:=\{\tau_{S}:=\prod_{j \in S}\tau_{j} \,;\, S \subset \{1,2,\ldots,K\}\}$ forms an orthonormal basis of the space of all functions on $\{-1,1\}^{K}$ endowed with
the scalar product induced by the uniform distribution
since $ \sharp {\mathscr W}_{K} =2^K = \dim \mathbb{R}^{\{-1,1\}^K} $.
\begin{Remark}
We can embed 
$\cup_{K \in {\mathbb N}}{\mathscr W}_{K}$ 
into $ L^2 ([0,1), \mathfrak{B} [0,1), \mathrm{Leb}) $, the space of
Borel measurable functions 
which are square integrable 
with respect to the Lebesgue measure, by
redefining $ \tau_i, i \in \mathbb{N} $, by
\begin{equation*}
    \tau_i (x) = \begin{cases}
    +1 & x \in \cup_{k=0}^{i-1} [\frac{2k}{2^i}, \frac{2k+1}{2^i}) \\
    -1 & \text{otherwise}
    \end{cases}
    , \quad x \in [0,1).
\end{equation*}
One can show that thus embedded
$\cup_{K \in {\mathbb N}}{\mathscr W}_{K}$ forms a complete orthonormal system of $L^{2}[0,1) $, which is often referred to as the {\em Walsh system} (\cite{Walsh}, \cite{Fine}).
\end{Remark}

To mimic the Brownian increments, we only use {\em odd} members of ${\mathscr W}_{K}$.

\begin{Proposition}
\label{prop:2.2}
Let $\varphi \equiv \varphi^{(K)}: \{1,2,\ldots,2^{K-1}\} \to {\mathcal O}_{K}:=\{\tau_{S} \,;\, S \subset \{1,2,\ldots,K\}, \sharp S \text{ is odd}\}$ be a bijection, and set 
\begin{align}
\label{eq:2.8}
(\Delta Z^{(n)})^{j}:=\varphi(j)\sqrt[]{\frac{T}{n}}, \quad j=1,2,\ldots,d.
\end{align}
Then the $d$-dimensional random variable $\Delta Z^{(n)}$ satisfies \eqref{eq:2.2}, \eqref{eq:2.3} and \eqref{eq:2.4}.
\end{Proposition}
\begin{proof}
\eqref{eq:2.2} is clear by the independence, and \eqref{eq:2.3} also clear by \eqref{eq:2.7}.
Let $j_{1}$, $j_{2}$, $j_{3} \in \{1,2,\cdots,d\}$.
Set $\tau_{S_{1}}:=\varphi(j_{1})$, $\tau_{S_{2}}:=\varphi(j_{2})$ and $\tau_{S_{3}}:=\varphi(j_{3})$ for each corresponding $S \subset \{1,2,\ldots,K\}$ such that $\sharp S$ is odd.
Then we obtain
\begin{align*}
\varphi(i_{1})\varphi(i_{2})\varphi(i_{3})
=\tau_{S_{1} \ominus S_{2}}\tau_{S_{1} \cap S_{2}}^{2}\tau_{S_{3}}
=\tau_{S_{1} \ominus S_{2}} \tau_{S_{3}}.
\end{align*}
On the other hand, $S_{1} \ominus S_{2}$ is never equal to $S_{3}$ since
$$
\sharp(S_{1} \ominus S_{2})=\sharp S_{1}+\sharp S_{2}-2 \sharp (S_{1} \cap S_{2})
$$
is even and $\sharp S_{3}$ is odd.
Hence \eqref{eq:2.4} holds by \eqref{eq:2.7}.
\end{proof}

\begin{Remark}
\label{rem:2.1}
We can also obtain the higher moments by considering the atom.
For $S \subset \{1,2,\ldots,K\}$, we set $S^{1}:=S$ and $S^{-1}:=S^{c} \equiv \{1,2,\ldots,K\} \setminus S$.
Let $p \in {\mathbb N}$.
Then we obtain
\begin{align*}
\prod_{k=1}^{p}\tau_{S_{k}}
=\prod_{(i_{1},i_{2},\ldots,i_{p}) \in \{-1,1\}^{p}}\tau_{\cap_{k=1}^{p}S_{k}^{i_{k}}}^{\sharp\{i_{k}\,;\,i_{k}=1\}}
=\prod_{\substack{(i_{1},i_{2},\ldots,i_{p}) \in \{-1,1\}^{p} \\ \sharp\{i_{k}\,;\,i_{k}=1\} \text{ is odd}}}\tau_{\cap_{k=1}^{p}S_{k}^{i_{k}}}^{\sharp\{i_{k}\,;\,i_{k}=1\}},
\end{align*}
where $\sharp\{i_{k}\,;\,i_{k}=1\}=\sum_{k=1}^{p}{\bf 1}_{\{1\}}(i_{k})$ and $\tau_{\emptyset} \equiv 1$.
Thus by the independence, we have
\begin{align*}
{\mathbb E}\left[\prod_{k=1}^{p}\tau_{S_{k}}\right]
&=\prod_{\substack{(i_{1},i_{2},\ldots,i_{p}) \in \{-1,1\}^{p} \\ \sharp\{i_{k}\,;\,i_{k}=1\} \text{ is odd}}}{\mathbb E}\left[\tau_{\cap_{k=1}^{p}S_{k}^{i_{k}}}^{\sharp\{i_{k}\,;\,i_{k}=1\}}\right]
=\prod_{\substack{(i_{1},i_{2},\ldots,i_{p}) \in \{-1,1\}^{p} \\ \sharp\{i_{k}\,;\,i_{k}=1\} \text{ is odd}}}\prod_{l \in \cap_{k=1}^{p}S_{k}^{i_{k}}}{\mathbb E}\left[\tau_{l}^{\sharp\{i_{k}\,;\,i_{k}=1\}}\right] \\
&=
\left\{
\begin{array}{ll}
1 & \displaystyle{\bigcup_{\substack{(i_{1},i_{2},\ldots,i_{p}) \in \{-1,1\}^{p} \\ \sharp\{i_{k}\,;\,i_{k}=1\} \text{ is odd}}}\bigcap_{k=1}^{p}S_{k}^{i_{k}}=\emptyset} \\ 
0 & \text{otherwise}.
\end{array}
\right.
\end{align*}
\end{Remark}

As a practical scheme, the bijection $\varphi$ should be algorithmically efficient in some sense, which we formulate mathematically as follows: Let
$$
{\mathcal O}_{k}:=\{\tau_{S} \,;\, S \subset \{1,2,\ldots,k\}, \sharp S \text{ is odd}\}, \quad k \in \{1,2,\ldots,K\}.
$$
\begin{Definition}
\label{def:2.3}
A map $\varphi: \{1,2,\ldots,2^{K-1}\} \to {\mathcal O}_{K}$ is called {\bf odd-ordered} if it satisfies $\varphi(\{1,2,\ldots,2^{k-1}\})={\mathcal O}_{k}$ for any $k \in \{1,2,\ldots,K\}$.
\end{Definition}
Note that this definition implies that an odd-ordered map is bijective. 
Now we give an explicit odd-ordered map. 
We inductively define a map $\varphi$ as follows:
\begin{align}
\label{odd-ordered map}
\varphi(k):= 
\left\{
\begin{array}{ll}
\tau_{1} & k=1 \\
\varphi(k-1)\tau_{1}\tau_{\theta(k)} & k \in \{2,3,\ldots,2^{K-1}\},
\end{array}
\right.
\end{align}
where
$$
\theta(k)= 
\left\{
\begin{array}{ll}
2 & k \text{ is even} \\
\max\{l \in {\mathbb N} \,;\, (k-1)2^{-l} \in {\mathbb N}\}+2 & k \text{ is odd}.
\end{array}
\right.
$$
Note that this map $\varphi$ can be regarded as a Gray code (\cite{Gray}), where each word differs from the next one in only one digit (each word has a Hamming distance of $1$ from the next word).
In Section \ref{sec:4}, we compare 
the computational time between 
a simpler algorithm and the algorithm based on this map $ \varphi$. The results show that in higher dimensions the latter one saves the computation time significantly.
\begin{Example}
\label{ex:2.1}
We inductively obtain
$\varphi(1)=\tau_{1}$, $\varphi(2)=\varphi(1)\tau_{1}\tau_{2}=\tau_{2}$, $\varphi(3)=\varphi(2)\tau_{1}\tau_{3}=\tau_{1}\tau_{2}\tau_{3}$, $\varphi(4)=\varphi(3)\tau_{1}\tau_{2}=\tau_{3}$ and so on.
\end{Example}

\begin{Theorem}
\label{theo:2.4}
The map $\varphi$ given above is odd-ordered.
\end{Theorem}
\begin{proof}
We inductively define a map $\psi:\{1,2,\ldots,2^{K}\} \to \{\tau_{S}; S \subset \{1,2,\ldots,K\}\}$ as follows:
$$
\psi(k):=
\left\{
\begin{array}{ll}
1 & k=1 \\
\psi(k-1)\tau_{\eta(k)} & k \in \{2,3,\ldots,2^{K}\},
\end{array}
\right.
$$
where
$$
\eta(k)= 
\left\{
\begin{array}{ll}
1 & k \text{ is even} \\
\max\{l \in {\mathbb N} \,;\, (k-1)2^{-l} \in {\mathbb N}\}+1 & k \text{ is odd}.
\end{array}
\right.
$$
Observe first that, for $k \in \{2,3,\ldots,2^{K-1}\}$, we obtain
$$
\psi(2k)=\psi(2k-1)\tau_{1}=\psi(2k-2)\tau_{1}\tau_{\eta(2k-1)}
$$
and
\begin{align*}
\eta(2k-1)&=\max\{l \in {\mathbb N} \,;\, (k-1)2^{-l+1} \in {\mathbb N}\}+1 \\
&=\max\{l \in {\mathbb N} \,;\, (k-1)2^{-l} \in {\mathbb N}\}+2 \\
&=\theta(k).
\end{align*}
Thus we inductively have
\begin{align}
\label{eq:2.9}
\psi(2k)=\varphi(k), \quad k \in \{1,2,\ldots,2^{K-1}\}.
\end{align}

Next, we see that, 
for $l \in \{2,3,\ldots,K\}$ and $k \in \{1,3,\ldots,2^{l-1}-1\}$,  
\begin{align*}
\eta(2^{l-1}+k)&=
\left\{
\begin{array}{ll}
\max\{l' \in {\mathbb N} \,;\, 2^{(l-l')-1} \in {\mathbb N}\}+1 & k=1 \\
\max\{l' \in {\mathbb N} \,;\, (2^{l-1}+k-1)2^{-l'} \in {\mathbb N}\}+1 & k \geq 3
\end{array}
\right. \\
&=
\left\{
\begin{array}{ll}
l & k=1 \\
\max\{l' \in {\mathbb N} \,;\, (k-1)2^{-l'} \in {\mathbb N}\}+1 & k \geq 3.
\end{array}
\right.
\end{align*}
This implies that
$$
\eta(2^{l-1}+k)=
\left\{
\begin{array}{ll}
l & k=1 \\
\eta(k) &  k \in \{2,3,\ldots,2^{l-1}\},
\end{array}
\right. \quad l \in \{2,3,\ldots,K\}.
$$
Hence for $l \in \{3,4,\ldots,K\}$ and $k \in \{1,2,\ldots,2^{l-1}\}$, we obtain
\begin{align*}
\psi(2^{l-1}+k)&=\psi(2^{l-1}+k-1)\tau_{\eta(2^{l-1}+k)}
=\psi(2^{l-1}+k-1)\tau_{\eta(k)} \\
&=\psi(2^{l-1}+k-2)\tau_{\eta(k)}\tau_{\eta(k-1)}
=\cdots \\
&=\psi(2^{l-1})\tau_{\eta(k)}\tau_{\eta(k-1)} \cdots \tau_{\eta(2)}\tau_{l}
=\psi(2^{l-1})\psi(k)\tau_{l}.
\end{align*}
Moreover, 
$$
\psi(2^{l-1})=\psi(2^{(l-1)-1}+2^{l-2})=\psi(2^{l-2})\psi(2^{l-2})\tau_{l-1}=\tau_{l-1}.
$$
Thus we have
\begin{align}
\label{eq:2.10}
\psi(2^{l-1}+k)=\tau_{l}\tau_{l-1}\psi(k), \quad l \in \{2,3,\ldots,K\}.
\end{align}
Consequently, we can inductively prove that the restricted map $\psi|_{\{1,2,\ldots,2^{l}\}}: \{1,2,\ldots,2^{l}\} \to \{\tau_{S} \,;\, S \subset \{1,2,\cdots,l\}\}$ is bijective by \eqref{eq:2.10}.
Furthermore, we can inductively show
$$
\{\psi(2k) \,;\, k \in \{1,2,\cdots,2^{k-1}\}\}={\mathcal O}_{k}.
$$
This implies that $\varphi$ is odd-ordered by \eqref{eq:2.9}.
\end{proof}

\section{Error Estimates Depending on Dimension}
\label{sec:3}
In this section, we discuss the error estimate for the weak convergence of the Euler Maruyama approximation given by $(\Delta Z_{\ell}^{(n)})_{\ell=1}^{n}$.
In particular, we will find the weak order of convergence with respect to the number of time steps.
This will imply that the EM schemes based on the Haar system \eqref{eq:2.5} and the Walsh system \eqref{eq:2.8}  have the same weak order $1$ of convergence as the standard EM scheme based on the Gaussian system.
We assume that the coefficients $\sigma$ and $b$ of the $d$-dimensional SDE ${\rm d}X_{t}=\sigma(X_{t}){\rm d}W_{t}+b(X_{t}){\rm d}t$, $t \geq 0$ satisfy the following assumption.

\begin{Assumption}
\label{ass:3.1}
The coefficients $\sigma$ and $b$ satisfy the following conditions.
\begin{itemize}
\item[{\bf A1-1.}] There exists a positive constant $C$ such that for any $x,y \in {\mathbb R}^{d}$, $|\sigma(y)-\sigma(x)| \vee |b(y)-b(x)| \leq C|y-x|$.
\item[{\bf A1-2.}] There exists a positive constant $C$ such that for any $y \in {\mathbb R}^{d}$, $|b(y)| \vee |\sigma(y)| \leq C(1+|y|)$.
\item[{\bf A1-3.}] For any $i,j \in \{1,2,\ldots,d\}$, $\sigma_{j}^{i}, b^{i} \in C_{P}^{4}({\mathbb R}^{d})$.
\end{itemize}
\end{Assumption}

Let $(X,W)$ be a solution of the $d$-dimensional SDE \eqref{eq:1.1} with the initial $X_{0} \equiv x_{0} \in {\mathbb R}^{d}$ on a complete probability space $(\Omega,{\mathcal F},{\mathbb P})$ with a filtration $({\mathcal F}_{t})_{t \geq 0}$.
Let $T>0$, $n \in {\mathbb N}$ be the number of partitions of the closed interval $[0,T]$ and $t_{\ell}:=\ell T/n$, $\ell \in \{1,2,\ldots,n\}$ be the equidistant
grid generated by equal time steps
on $[0,T]$.
The flow associated with the SDE is defined on the same filtered probability space since the Lipschitz condition {\bf A1-1} provides a unique strong solution to the SDE.
Note that the condition {\bf A1-1} is only used to guarantee a unique strong solution and not in the discussion about the error estimate.
We consider the Euler-Maruyama approximation $X^{(n)}$ of the equation \eqref{eq:1.1} given by $d$-dimensional random variables $(\Delta Z_{\ell}^{(n)})_{\ell=1}^{n}$ which satisfy the following assumption.

\begin{Assumption}
\label{ass:3.2}
The $d$-dimensional random variables $(\Delta Z_{\ell}^{(n)})_{\ell=1}^{n}$ satisfy the following conditions.
\begin{itemize}
\item[{\bf A2-1.}] $(\Delta Z_{\ell}^{(n)})_{\ell=1}^{n}$ is $({\mathcal F}_{t_{\ell}})_{\ell=1}^{n}$-adapted.
\item[{\bf A2-2.}] For any $\ell \in \{1,2,\ldots,n\}$, $\Delta Z_{\ell}^{(n)}$ and ${\mathcal F}_{t_{\ell-1}}$ are independent.
\item[{\bf A2-3.}] For any $j_{1},j_{2},j_{3} \in \{1,2,\ldots,d\}$ and $\ell \in \{1,2,\ldots,n\}$,
$$
{\mathbb E}\left[(\Delta Z_{\ell}^{(n)})^{j_{1}}\right]=0, \quad {\mathbb E}\left[(\Delta Z_{\ell}^{(n)})^{j_{1}}(\Delta Z_{\ell}^{(n)})^{j_{2}}\right]=\frac{T}{n}\delta_{j_{1},j_{2}}, \quad {\mathbb E}\left[(\Delta Z_{\ell}^{(n)})^{j_{1}}(\Delta Z_{\ell}^{(n)})^{j_{2}}(\Delta Z_{\ell}^{(n)})^{j_{3}}\right]=0.
$$
\end{itemize}
\end{Assumption}

Obviously, the Gaussian system satisfies Assumption \ref{ass:3.2}.
We see that the Haar system \eqref{eq:2.5} and the Walsh system \eqref{eq:2.8} also satisfy it from Proposition \ref{prop:2.1} and \ref{prop:2.2} in Section \ref{sec:2}.
Also note that we do not assume the independence for each component of the vector $\Delta Z_{\ell}^{(n)}$ here.
That is, we do not know whether $(\Delta Z_{\ell}^{(n)})^{1}, (\Delta Z_{\ell}^{(n)})^{2}, \ldots, (\Delta Z_{\ell}^{(n)})^{d}$ are independent or not under Assumption \ref{ass:3.2}.

Moreover, we set for any $p \in {\mathbb N} \cup \{0\}$,
$$
M_{p}^{(n)}(Z):=\max_{\ell \in \{1,2,\ldots,n\}}{\mathbb E}\left[\left|\Delta Z_{\ell}^{(n)}\right|^{p}\right],
$$
where $|x|=(\sum_{i=1}^{d}|x^{i}|^{2})^{1/2}$, $x \in {\mathbb R}^{d}$.

Under these assumptions, we can get the following error estimate, where we can see the difference of the bounds of the error with different choices of $(\Delta Z_{\ell}^{(n)})_{\ell=1}^{n}$.

\begin{Theorem}
\label{theo:3.3}
Suppose that Assumption \ref{ass:3.1} holds.
Then for any $f \in C_{P}^{4}({\mathbb R}^{d})$, 
we have the following estimates.
(i) For any $ (\Delta Z^{(n)}_\ell)_{\ell=1}^n $ which satisfies Assumption \ref{ass:3.2}, 
there exists a positive constant $C_1$
that is independent of $ n $ 
and $ (\Delta Z^{(n)}_\ell)_{\ell=1}^n $ such that 
$$
\left|{\mathbb E}[f(X_{T})]-{\mathbb E}[f(X_{T}^{(n)})]\right| \leq C_{r}^{(n)}(Z)\left(nM_{8}^{(n)}(Z)^{1/2}+\frac{1}{n}\right),
$$
where
$$
C_{r}^{(n)}(Z)=C_1\left(1+M_{4r}^{(n)}(Z)^{1/2}\right)\exp\left\{C_1\left(1+n\left(M_{2}^{(n)}(Z) \vee M_{2(r+2)}^{(n)}(Z) \vee M_{2(4r+3)}^{(n)}(Z)\right)\right)\right\},
$$
$$
r:=\min\left\{r \in {\mathbb N}\cup \{0\} \,;\, 
\begin{array}{ll}
\text{There exists a positive constant } C \text{ such that for any } y \in {\mathbb R}^{d}, \\
\displaystyle{\max_{\substack{i_{1},i_{2},\ldots,i_{k} \in \{1,2,\ldots,d\} \\ k \in \{1,2,3,4\}}}\left|\frac{\partial^{k}f(y)}{\partial y^{i_{1}}\partial y^{i_{2}}\cdots\partial y^{i_{k}}}\right|} \\
\displaystyle{\vee \max_{\substack{i \in \{1,2,\ldots,d\} \\ j \in \{0,1,\ldots,d\}}}\max_{\substack{i_{1},i_{2},\ldots,i_{k} \in \{1,2,\ldots,d\} \\ k \in \{1,2\}}}\left|\frac{\partial^{k}\sigma_{j}^{i}(y)}{\partial y^{i_{1}}\partial y^{i_{2}}\cdots\partial y^{i_{k}}}\right| \leq C\left(1+|y|^{2r}\right)}.
\end{array}
\right\}.
$$
and $\sigma_{0}:=b$.
(ii) If $ (\Delta Z^{(n)}_\ell), \ell=1,2, \ldots, d $ 
are the increment of a $d$-dimensional Brownian motion, then, 
if $ n \geq (8r+5) T d $, 
\begin{equation}\label{boundGauss}
\left|{\mathbb E}[f(X_{T})]-{\mathbb E}[f(X_{T}^{(n)})]\right| \leq 2C_1 e^{C_1 (1+dT)} (1+ T^2 K_{d}){n^{-1}},
\end{equation}
where $ K_d $ is a constant depending only on $ d $, which satisfies
\begin{equation*}
   \sqrt{ d (d+1)(d+2)(d+3)} \leq K_d \leq 
    d^2 \sqrt{105}.  
\end{equation*}
(iii) If $ (\Delta Z^{(n)}_\ell), \ell=1,2, \ldots, d $ 
are independent copies of the Haar system \eqref{eq:2.5}, then, if $ n \geq T 2^{K-1}$,   
we have 
\begin{equation}\label{Haarbound}
\left|{\mathbb E}[f(X_{T})]-{\mathbb E}[f(X_{T}^{(n)})]\right| \leq 2C_1 e^{C_1 (1+dT)} (1+ T^2 \sqrt{d 2^{3(K-1)}}){n^{-1}}.
\end{equation}
(iv) If $ (\Delta Z^{(n)}_\ell), \ell=1,2, \ldots, d $ 
are independent copies of the Walsh system \eqref{eq:2.8}, then   
we have, if $ n \geq T d $, 
\begin{equation}\label{Walshbound}
\left|{\mathbb E}[f(X_{T})]-{\mathbb E}[f(X_{T}^{(n)})]\right| \leq 2C_1 e^{C_1 (1+dT)} (1+ T^2 d^2){n^{-1}}.
\end{equation}
\end{Theorem}

The proof of Theorem \ref{theo:3.3} is postponed in Section \ref{sec:5} (Appendix).

\begin{Remark}
\label{rem:3.1}
The above theorem 
shows that, 
not only weak order $1$ of convergence, but also the difference of the bounds with the different choices 
of $(\Delta Z_{\ell}^{(n)})_{\ell=1}^{n}$.
We observe that 
the bound in the Gaussian scheme is the largest. 
Note that, in the Haar scheme and the Walsh scheme, 
$ K $ is chosen so that 
$ d \leq 2^{K-1} $, and practically 
the minimum of such $ K $ will be chosen
although the bound of the Walsh scheme is independent of $ K $.
In the Haar scheme, 
for a large $ d $, 
$ 2^{K-1} -d $ can also be large
(in worst case it is $ 2^{K-1}-2^{K-2} -1= d -2 $).
In such a case the scheme based on the Haar system may suffer a relatively slow convergence, compared with the Walsh scheme. 
In fact, although the bounds are not tight in general, the 
following extreme but simple example shows that 
in some cases the difference of the bounds is almost optimal in
respect of $ d $, though 
we can not see the difference with the Gaussian scheme by the example.
\end{Remark}

In the following proposition, we consider the simple case where $b \equiv 0$, $\sigma$ is the unit matrix, $x_{0}=0$ and $f(x)=|x|^{4}$.

\begin{Proposition}
Let $ W \equiv (W^1,W^2, \ldots, W^d) $ be a $ d $-dimensional Brownian motion starting from $ 0 $.
(i) If $ (\Delta Z^{(n)}_{\ell}), \ell=1,2, \ldots, d $ 
are independent copies of the Haar system \eqref{eq:2.5}, then
\begin{equation*}
    \mathbb{E} \left[\left(\sum_{i=1}^d (W^i_T)^2 \right)^2
    \right]
    -  \mathbb{E} \left[\left( \sum_{i=1}^d\left(\sum_{\ell=1}^n (\Delta Z_\ell^{(n)})^i\right)^2\right)^2\right] = \left(2d + d (d-2^{K-1}) \right)T^{2}n^{-1}.
\end{equation*}
(ii) If $ (\Delta Z^{(n)}_\ell), \ell=1,2, \ldots, d $ 
are independent copies of the Walsh system \eqref{eq:2.8}, then 
\begin{equation}\label{prop:3.new}
    \mathbb{E} \left[\left(\sum_{i=1}^d (W^i_T)^2\right)^2
    \right]
    -  \mathbb{E} \left[\left( \sum_{i=1}^d\left(\sum_{\ell=1}^n (\Delta Z_\ell^{(n)})^i\right)^2\right)^2\right] = 2dT^{2}n^{-1}.
\end{equation}
\end{Proposition}
\begin{proof}
Observe first that, in both cases, 
\begin{equation*}
    \begin{split}
        \mathbb{E}\left[\left( \sum_{i=1}^d\left(\sum_{\ell=1}^n (\Delta Z_\ell^{(n)})^i\right)^2\right)^2\right]
        &=  \sum_{i=1}^d\left(\sum_{\ell=1}^n\mathbb{E}\left[
        ((\Delta Z^{(n)}_\ell)^i)^4\right]
        +3\sum_{\ell \ne \ell'}\mathbb{E}\left[((\Delta Z^{(n)}_\ell)^i)^2\right]\mathbb{E}\left[((\Delta Z^{(n)}_{\ell'})^i)^2\right]\right)  \\
        &\hspace{0.38cm}+ \sum_{i\ne i'}\left(
        \sum_{\ell=1}^n 
       \mathbb{E}\left[( (\Delta Z^{(n)}_\ell)^i)^2((\Delta Z^{(n)}_\ell)^{i'})^2\right]
        + \sum_{\ell \ne \ell'}\mathbb{E}\left[((\Delta Z^{(n)}_\ell)^i)^2\right]\mathbb{E}\left[((\Delta Z^{(n)}_{\ell'})^{i'})^2\right]\right) \\
        &=\sum_{i=1}^{d}\left(\sum_{\ell=1}^{n}{\mathbb E}\left[((\Delta Z_{\ell}^{(n)})^{i})^{4}\right]+3n(n-1)\left(\frac{T}{n}\right)^{2}\right) \\
        &\hspace{0.38cm}+\sum_{i \neq i'}\left(\sum_{\ell=1}^{n}{\mathbb E}\left[((\Delta Z_{\ell}^{(n)})^{i})^{2}((\Delta Z^{(n)}_{\ell})^{i'})^{2}\right]+n(n-1)\left(\frac{T}{n}\right)^{2}\right).
    \end{split}
\end{equation*}
In (i), we have 
\begin{equation*}
    \mathbb{E} \left[ ((\Delta Z^{(n)}_\ell)^i)^4\right] = 2^{K-1} \left(\frac{T}{n}\right)^2 \quad \text{ and } \quad \mathbb{E} \left[
       ( (\Delta Z^{(n)}_\ell)^i)^2((\Delta Z^{(n)}_\ell)^{i'})^2  
       \right] = 0
\end{equation*}
if $ i \ne i'$.
In (ii), we have 
\begin{equation*}
    \mathbb{E} \left[ ((\Delta Z^{(n)}_\ell)^i)^4\right] =  \left(\frac{T}{n}\right)^2 \quad \text{ and } \quad \mathbb{E} \left[
       ( (\Delta Z^{(n)}_\ell)^i)^2((\Delta Z^{(n)}_\ell)^{i'})^2  
       \right] = \left(\frac{T}{n}\right)^2
\end{equation*}
if $ i \ne i'$.
Since 
\begin{equation*}
    \mathbb{E} \left[\left(\sum_{i=1}^d (W^i_T)^2\right)^2
    \right] = (d^2 + 2d) T^2, 
\end{equation*}
we have the desired results. 
\end{proof}

\section{Numerical experiments}
\label{sec:4}
In this section, we compare the quality among the Euler-Maruyama schemes $X^{(n)}$ by the Gaussian system, the Haar system \eqref{eq:2.5} and the Walsh system \eqref{eq:2.8} thorough some numerical experiments using the Monte Carlo method for ${\mathbb E}[f(X_{T}^{(n)})]$.
First, we compare the confidence interval, the computational time, the sample variance and the ratio standard deviation/computational time among them under the following situations Case 1 and Case 2.
\begin{itemize}
\item Case 1. We consider the $d$-dimensional SDE
$$
{\rm d}X_{t}^{i}=\frac{1}{d}\sum_{j=1}^{d}(X_{t}^{i}-X_{t}^{j}){\rm d}t+{\rm d}W_{t}^{i}, \quad t \geq 0, \quad i \in \{1,2,\dots,d\}
$$
with the initial $x_{0}=(1,1,\ldots,1)$, the time horizon $T=1$ and the test function $f(x)=\cos(\sum_{i=1}^{d}x^{i})$.
\item Case 2. We consider the $d$-dimensional SDE \eqref{eq:1.1} with the coefficients $b \equiv 0$,
$$
\sigma_{j}^{i}(x)=
\left\{
\begin{array}{ll}
x^{i-1} & j=i-1 \\
x^{i} & j=i \\
x^{i+1} & j=i+1 \\
0 & \text{otherwise},
\end{array}
\right. \quad x \in {\mathbb R}^{d}, \quad i,j \in \{1,2,\ldots,d\}
$$
and the initial $x_{0}=(1,1,\ldots,1)$, the time horizon $T=1$ and the test function $f(x)=\cos(\sum_{i=1}^{d}x^{i})$.
\end{itemize}

In order to see 
the effect of the odd-ordered map 
given by \eqref{odd-ordered map}, 
we compare it with the following simple algorithm 
which is not odd-ordered. 
Let $K \in \{1,2,\ldots,32\}$ and $S \subset \{1,2,\ldots,K\}$.
By setting the bits corresponding to the numbers of the elements of $S$ to $1$ and the rest to $0$, $\tau_{S}$ can be considered to correspond to an element of the bit mask.
By using the bit shift operator, 
we construct
$\varphi$ by skipping even 
elements like $\varphi(1)=\tau_{1}=(0,\ldots,0,1)$, $\varphi(2)=\tau_{2}=(0,\ldots,0,1,0)$, $\ldots$, $\varphi(K)=\tau_{K}=(1,0,\ldots,0)$, $\varphi(K+1)=\tau_{1}\tau_{2}\tau_{3}=(0,\ldots,0,1,1,1)$, $\varphi(K+2)=\tau_{1}\tau_{2}\tau_{4}=(0,\ldots,0,1,0,1,1)$, $\ldots$
and so on.
Note that this $ \varphi $ is not an odd-ordered map for $ K \geq 4 $; an odd-ordered map should satisfy 
$ \{\varphi (3), \varphi (4) \} = \{ \tau_3, \tau_1\tau_2 \tau_3\} $ (see Definition \ref{def:2.3}).

In the figures and tables (A1)-(G2) below, we describe some numerical results using the Monte Carlo method for ${\mathbb E}[f(X_{T}^{(n)})]$.
The figures and tables (A1)-(G1) and (A2)-(G2) show results for Case 1 and Case 2, respectively.
Recall that $d$, $n$ and $m$ mean the dimension of the SDEs, the number of time steps of the EM schemes and the number of Monte Carlo trials, respectively.
We chose $\min\{K \in {\mathbb N}\,;\,d \leq 2^{K-1}\}$ as $K \in {\mathbb N}$ in Section \ref{sec:2.2} and \ref{sec:2.3} for all the results using the Haar system \eqref{eq:2.5} and the Walsh system \eqref{eq:2.8} to stabilize the numerical calculations as much as possible.
Specifically, we chose $K=6$ for $d=32=2^{5}$, $K=7$ for $d=40,48,56,64=2^{6}$, $K=8$ for $d=80,96,112,128=2^{7}$ and $K=9$ for $d=160,192,224,256=2^{8}$.
In (A1) and (A2), the $x$-axis indicates the number of Monte Carlo trials, and the $y$-axis indicates the value of the sample mean of ${\mathbb E}[f(X_{T}^{(n)})]$.
In the other figures, the $x$-axis indicates the value of the dimension $d=32=2^{5}, 40,48,56,64=2^{6}, 80,96,112,128=2^{7}, 160,192,224,256=2^{8}$ of the SDEs, and we compare the EM schemes about various items.
The purple lines and the green lines are results for the EM schemes by the Gaussian system and the Haar system, respectively.
The blue lines named Walsh system 1 are results for the EM scheme by the Walsh system with the algorithm created using the odd-ordered map \eqref{odd-ordered map}.
On the other hand, the yellow lines named Walsh system 2 are results for the EM scheme by the Walsh system with the simple algorithm created using the bit mask.

The figures (A1) and (A2) show numerical results about the sample mean of ${\mathbb E}[f(X_{T}^{(n)})]$ with $d=100$ and $n=2^{5}$.
In the figures and tables (B1), (B2), (C1) and (C2), we compare, among the EM schemes, the $95\%$ confidence interval for each dimension with $n=2^{10}$ and $m=2^{20}$.
Here, we used
\begin{align*}
&\text{sample mean}-2.262 \times \sqrt[]{\frac{\text{unbiased variance}}{m}} \leq \text{population mean} \\
&\hspace{5.0cm}\leq \text{sample mean}+2.262 \times \sqrt[]{\frac{\text{unbiased variance}}{m}}
\end{align*}
to calculate the $95\%$ confidence interval since we do not know the population variance, and we used $2 \times 2.262 \times \sqrt[]{\frac{\text{unbiased variance}}{m}}$ for the size of the $95\%$ confidence interval.
From these results, we can confirm that the EM schemes by the Haar system and the Walsh system converge with the same accuracy as the standard EM scheme by the Gaussian scheme.

In the figures (D1), (D2), (E1) and (E2), we compare, among the EM schemes, the computational time (sec) of Monte Calro trials for each dimension with $n=2^{10}$ and $m=2^{20}$.
The figures (E1) and (E2) are enlargements of (D1) and (D2), respectively, excluding the standard EM scheme with the Gaussian system.
These results show that the Haar system has the fastest computation time, followed the Walsh system by the odd-ordered map, the Walsh system by the bit mask and the Gaussian system.
It can be seen that the higher the dimension, the more the computation time can be greatly improved by using the Haar system and the Walsh system.

In the figures (F1) and (F2), we compare, among the EM schemes, the sample variance of ${\mathbb E}[f(X_{T}^{(n)})]$ for each dimension with $n=2^{10}$ and $m=2^{20}$.
We can confirm that the EM schemes by the Haar and Walsh systems are as stable as the standard EM scheme by the Gaussian scheme.
These results can also be seen in the figures (B1) and (B2).

In the figures (G1) and (G2), we compare, among the EM schemes, the ratio of the standard deviation of ${\mathbb E}[f(X_{T}^{(n)})]$ and the computational time (the standard deviation divided the computational time) for each dimension with $n=2^{10}$ and $m=2^{20}$.
These results also show that the Haar system has the highest accuracy, followed the Walsh system by the odd-ordered map, the Walsh system by the bit mask and the Gaussian system.

{\tiny
\begin{table}
\begin{tabular}{cc}
\begin{minipage}{0.5\hsize}
\begin{center}
\includegraphics[width=0.8\textwidth]{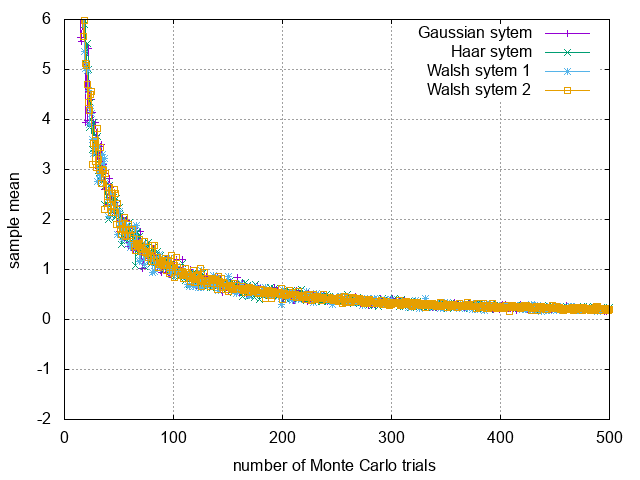}
\caption*{{\tiny (A1) sample mean in Case 1 with $d=100$ and $n=2^{5}$}}
\end{center}
\end{minipage}
\begin{minipage}{0.5\hsize}
\begin{center}
\includegraphics[width=0.8\textwidth]{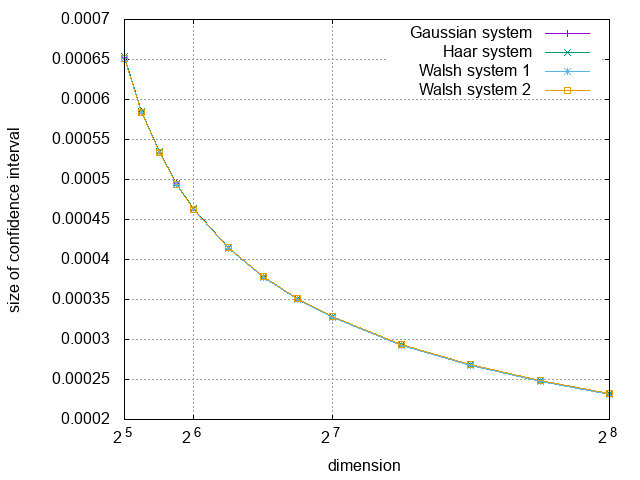}
\caption*{{\tiny (B1) size of $95\%$ confidence interval in Case 1 with $n=2^{10}$ and $m=2^{20}$}}
\end{center}
\end{minipage}
\end{tabular}
\begin{tabular}{|r|r|r|r|r|} \hline
$d$ & Gaussian system & Haar system & Walsh system 1 & Walsh system 2 \\ \hline \hline
$2^{5}=32$ & $[5.316311 \times 10^{-1}, 5.322835 \times 10^{-1}]$ & $[5.314268 \times 10^{-1}, 5.320800 \times 10^{-1}]$ & $[5.317695 \times 10^{-1}, 5.324213 \times 10^{-1}]$ & $[5.314857 \times 10^{-1}, 5.321376 \times 10^{-1}]$ \\ \hline
$40$ & $[5.331991 \times 10^{-1}, 5.337837 \times 10^{-1}]$ & $[5.332854 \times 10^{-1}, 5.338704 \times 10^{-1}]$ & $[5.334386 \times 10^{-1}, 5.340226 \times 10^{-1}]$ & $[5.331589 \times 10^{-1}, 5.337433 \times 10^{-1}]$ \\ \hline
$48$ & $[5.343189 \times 10^{-1}, 5.348532 \times 10^{-1}]$ & $[5.343348 \times 10^{-1}, 5.348696 \times 10^{-1}]$ & $[5.345783 \times 10^{-1}, 5.351117 \times 10^{-1}]$ & $[5.342897 \times 10^{-1}, 5.348237 \times 10^{-1}]$ \\ \hline
$56$ & $[5.352447 \times 10^{-1}, 5.357398  \times 10^{-1}]$ & $[5.352226  \times 10^{-1}, 5.357179  \times 10^{-1}]$ & $[5.354836  \times 10^{-1},5.359779 \times 10^{-1}]$ & $[5.351636  \times 10^{-1}, 5.356580 \times 10^{-1}]$ \\ \hline
$2^{6}=64$ & $[5.359463 \times 10^{-1}, 5.364091 \times 10^{-1}]$ & $[5.357873 \times 10^{-1}, 5.362510 \times 10^{-1}]$ & $[5.360351 \times 10^{-1}, 5.364976 \times 10^{-1}]$ & $[5.357885 \times 10^{-1}, 5.362515 \times 10^{-1}]$ \\ \hline
$80$ & $[5.367704 \times 10^{-1},5.371847 \times 10^{-1}]$ & $[5.366489 \times 10^{-1}, 5.370637 \times 10^{-1}]$ & $[5.368968 \times 10^{-1}, 5.373108 \times 10^{-1}]$ & $[5.366503 \times 10^{-1}, 5.370649 \times 10^{-1}]$ \\ \hline
$96$ & $[5.373508 \times 10^{-1}, 5.377292 \times 10^{-1}]$ & $[5.372468 \times 10^{-1}, 5.376258 \times 10^{-1}]$ & $[5.374371 \times 10^{-1}, 5.378149 \times 10^{-1}]$ & $[5.372517 \times 10^{-1}, 5.376304 \times 10^{-1}]$ \\ \hline
$112$ & $[5.377219 \times 10^{-1}, 5.380727 \times 10^{-1}]$ & $[5.376723 \times 10^{-1}, 5.380233 \times 10^{-1}]$ & $[5.378768 \times 10^{-1}, 5.382268 \times 10^{-1}]$ &  $[5.376798 \times 10^{-1}, 5.380305 \times 10^{-1}]$ \\ \hline
$2^{7}=128$ & $[5.380441 \times 10^{-1}, 5.383721 \times 10^{-1}]$ & $[5.379881 \times 10^{-1}, 5.383166 \times 10^{-1}]$ & $[5.382057 \times 10^{-1}, 5.385332 \times 10^{-1}]$ &  $[5.379877 \times 10^{-1}, 5.383158 \times 10^{-1}]$ \\ \hline
$160$ & $[5.384504 \times 10^{-1}, 5.387440 \times 10^{-1}]$ & $[5.384511 \times 10^{-1}, 5.387447 \times 10^{-1}]$ & $[5.386179 \times 10^{-1}, 5.389110 \times 10^{-1}]$ &  $[5.384536 \times 10^{-1}, 5.387472 \times 10^{-1}]$ \\ \hline
$192$ & $[5.387672 \times 10^{-1}, 5.390351 \times 10^{-1}]$ & $[5.387074 \times 10^{-1}, 5.389756 \times 10^{-1}]$ & $[5.388499 \times 10^{-1}, 5.391178 \times 10^{-1}]$ & $[5.387702 \times 10^{-1}, 5.390384 \times 10^{-1}]$ \\ \hline
$224$ & $[5.389770 \times 10^{-1}, 5.392251 \times 10^{-1}]$ & $[5.389362 \times 10^{-1}, 5.391847 \times 10^{-1}]$ & $[5.390752 \times 10^{-1}, 5.393231 \times 10^{-1}]$ &  $[5.389677 \times 10^{-1}, 5.392160 \times 10^{-1}]$ \\ \hline
$2^{8}=256$ & $[5.391502 \times 10^{-1}, 5.393822 \times 10^{-1}]$ & $[5.391090 \times 10^{-1}, 5.393414 \times 10^{-1}]$ & $[5.392262 \times 10^{-1}, 5.394581 \times 10^{-1}]$ & $[5.391671 \times 10^{-1}, 5.393993 \times 10^{-1}]$ \\ \hline
\end{tabular}
\caption*{{\tiny (C1) $95\%$ confidence interval in Case 1 with $n=2^{10}$ and $m=2^{20}$}}
\begin{tabular}{cc}
\begin{minipage}{0.5\hsize}
\begin{center}
\includegraphics[width=0.8\textwidth]{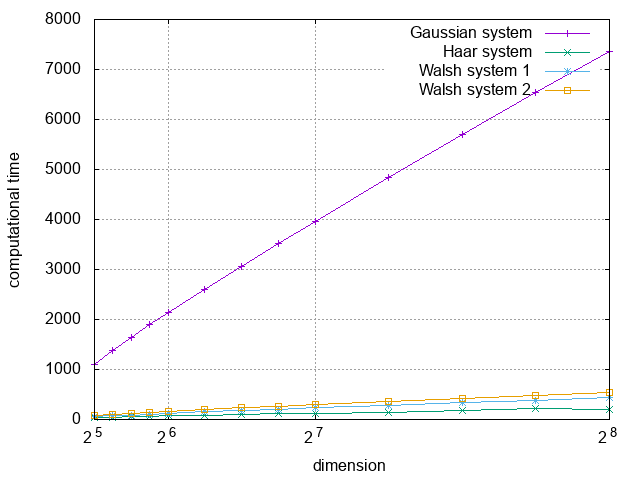}
\caption*{{\tiny (D1) computational time in Case 1 with $n=2^{10}$ and $m=2^{20}$}}
\end{center}
\end{minipage}
\begin{minipage}{0.5\hsize}
\begin{center}
\includegraphics[width=0.8\textwidth]{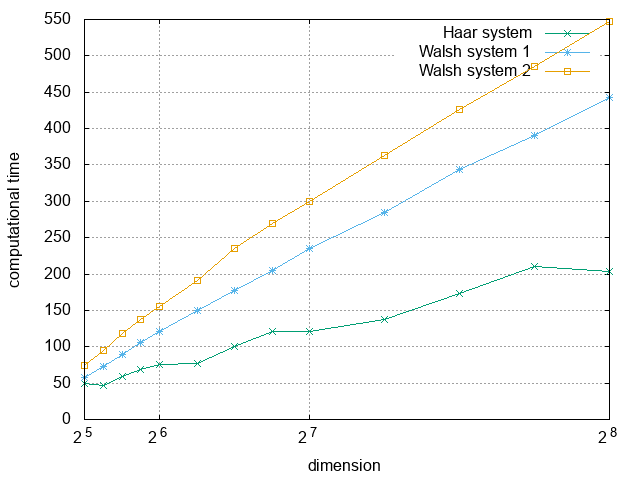}
\caption*{{\tiny (E1) computational time in Case 1 with $n=2^{10}$ and $m=2^{20}$}}
\end{center}
\end{minipage}
\end{tabular}
\begin{tabular}{cc}
\begin{minipage}{0.5\hsize}
\begin{center}
\includegraphics[width=0.8\textwidth]{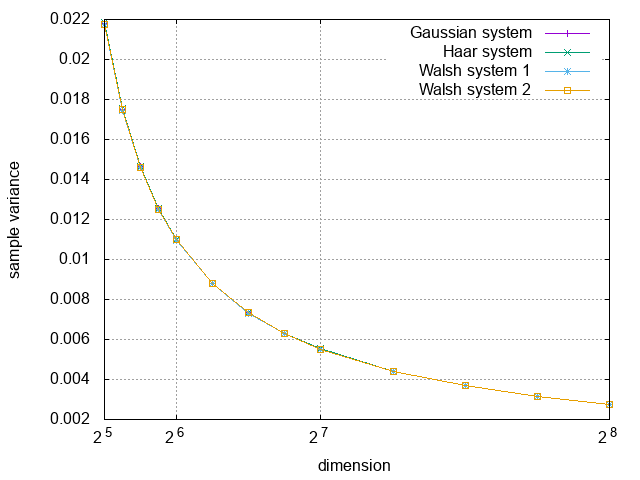}
\caption*{{\tiny (F1) sample variance in Case 1 with $n=2^{10}$ and $m=2^{20}$}}
\end{center}
\end{minipage}
\begin{minipage}{0.5\hsize}
\begin{center}
\includegraphics[width=0.8\textwidth]{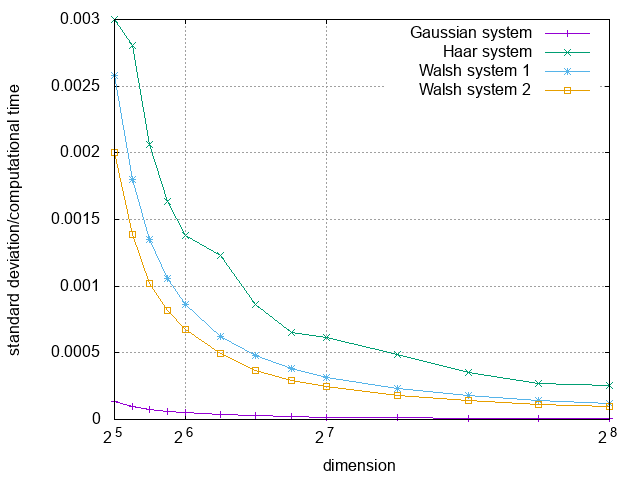}
\caption*{{\tiny (G1) standard deviation/computational time in Case 1 with $n=2^{10}$ and $m=2^{20}$}}
\end{center}
\end{minipage}
\end{tabular}
\end{table}
}

{\tiny
\begin{table}
\begin{tabular}{cc}
\begin{minipage}{0.5\hsize}
\begin{center}
\includegraphics[width=0.8\textwidth]{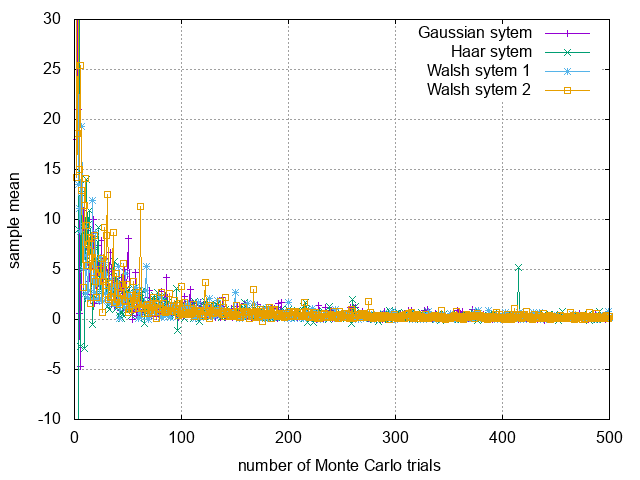}
\caption*{{\tiny (A2) sample mean in Case 2 with $d=100$ and $n=2^{5}$}}
\end{center}
\end{minipage}
\begin{minipage}{0.5\hsize}
\begin{center}
\includegraphics[width=0.8\textwidth]{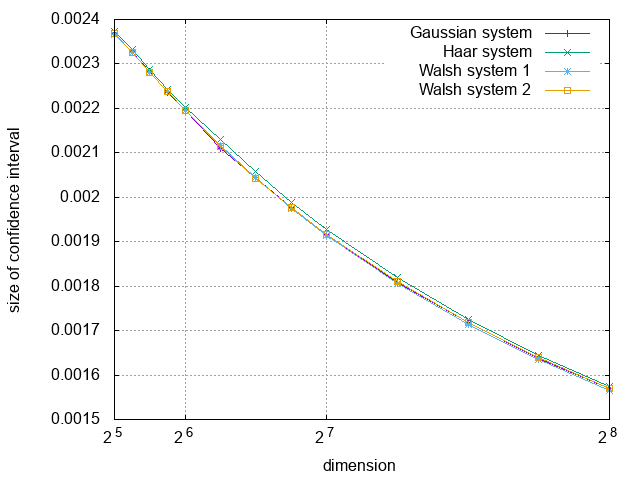}
\caption*{{\tiny (B2) size of $95\%$ confidence interval in Case 2 with $n=2^{10}$ and $m=2^{20}$}}
\end{center}
\end{minipage}
\end{tabular}
\begin{tabular}{|r|r|r|r|r|} \hline
$d$ & Gaussian system & Haar system & Walsh system 1 & Walsh system 2 \\ \hline \hline
$2^{5}=32$ & $[5.388813 \times 10^{-1}, 5.412508 \times 10^{-1}]$ & $[5.380646 \times 10^{-1}, 5.404370 \times 10^{-1}]$ & $[5.394858 \times 10^{-1}, 5.418530 \times 10^{-1}]$ &  $[5.393577 \times 10^{-1}, 5.417261 \times 10^{-1}]$ \\ \hline
$40$ & $[5.312468 \times 10^{-1}, 5.335721 \times 10^{-1}]$ & $[5.291805 \times 10^{-1}, 5.315129 \times 10^{-1}]$ & $[5.313038 \times 10^{-1}, 5.336298 \times 10^{-1}]$ & $[5.307647 \times 10^{-1}, 5.330906 \times 10^{-1}]$ \\ \hline
$48$ & $[5.248937 \times 10^{-1}, 5.271766  \times 10^{-1}]$ & $[5.236429  \times 10^{-1}, 5.259303  \times 10^{-1}]$ & $[5.257482  \times 10^{-1}, 5.280289 \times 10^{-1}]$ & $[5.255703  \times 10^{-1}, 5.278508  \times 10^{-1}]$ \\ \hline
$56$ & $[5.213830 \times 10^{-1}, 5.236191 \times 10^{-1}]$ & $[5.196167 \times 10^{-1}, 5.218595 \times 10^{-1}]$ & $[5.217672 \times 10^{-1}, 5.240042 \times 10^{-1}]$ & $[5.213875 \times 10^{-1}, 5.236256 \times 10^{-1}]$ \\ \hline
$2^{6}=64$ & $[5.190664 \times 10^{-1}, 5.212606 \times 10^{-1}]$ & $[5.167808 \times 10^{-1}, 5.189819 \times 10^{-1}]$ & $[5.184511 \times 10^{-1}, 5.206471 \times 10^{-1}]$ & $[5.183159 \times 10^{-1}, 5.205116 \times 10^{-1}]$ \\ \hline
$80$ & $[5.152587 \times 10^{-1}, 5.173693 \times 10^{-1}]$ & $[5.108301 \times 10^{-1}, 5.129611 \times 10^{-1}]$ & $[5.140594 \times 10^{-1}, 5.161762 \times 10^{-1}]$ & $[5.145808 \times 10^{-1}, 5.166955 \times 10^{-1}]$ \\ \hline
$96$ & $[5.120016 \times 10^{-1}, 5.140460 \times 10^{-1}]$ & $[5.089440 \times 10^{-1}, 5.110019 \times 10^{-1}]$ & $[5.119263 \times 10^{-1}, 5.139703 \times 10^{-1}]$ & $[5.123631 \times 10^{-1}, 5.144062 \times 10^{-1}]$ \\ \hline
$112$ & $[5.113176 \times 10^{-1}, 5.132954 \times 10^{-1}]$ & $[5.083539 \times 10^{-1}, 5.103426 \times 10^{-1}]$ & $[5.112470 \times 10^{-1}, 5.132221 \times 10^{-1}]$ & $[5.111189 \times 10^{-1}, 5.130968 \times 10^{-1}]$ \\ \hline
$2^{7}=128$ & $[5.105766 \times 10^{-1},5.124928 \times 10^{-1}]$ & $[5.079734 \times 10^{-1},5.099018 \times 10^{-1}]$ & $[5.108941 \times 10^{-1}, 5.128082 \times 10^{-1}]$ & $[5.107330 \times 10^{-1}, 5.126497 \times 10^{-1}]$ \\ \hline
$160$ & $[5.104821 \times 10^{-1}, 5.122913 \times 10^{-1}]$ & $[5.070743 \times 10^{-1}, 5.088945 \times 10^{-1}]$ & $[5.109406 \times 10^{-1}, 5.127476 \times 10^{-1}]$ & $[5.107402 \times 10^{-1}, 5.125503 \times 10^{-1}]$ \\ \hline
$192$ & $[5.111927 \times 10^{-1}, 5.129120 \times 10^{-1}]$ & $[5.084502 \times 10^{-1}, 5.101758 \times 10^{-1}]$ & $[5.116462 \times 10^{-1}, 5.133605 \times 10^{-1}]$ & $[5.112806 \times 10^{-1}, 5.129998 \times 10^{-1}]$ \\ \hline
$224$ & $[5.127379 \times 10^{-1}, 5.143750 \times 10^{-1}]$ & $[5.099301 \times 10^{-1}, 5.115749 \times 10^{-1}]$ & $[5.127729 \times 10^{-1}, 5.144076 \times 10^{-1}]$ & $[5.124041 \times 10^{-1}, 5.140432 \times 10^{-1}]$ \\ \hline
$2^{8}=256$ & $[5.131295 \times 10^{-1}, 5.146994 \times 10^{-1}]$ & $[5.114802 \times 10^{-1}, 5.130537 \times 10^{-1}]$ & $[5.140365 \times 10^{-1}, 5.156008 \times 10^{-1}]$ & $[5.135172 \times 10^{-1}, 5.150867 \times 10^{-1}]$ \\ \hline
\end{tabular}
\caption*{{\tiny (C2) $95\%$ confidence interval in Case 2 with $n=2^{10}$ and $m=2^{20}$}}
\begin{tabular}{cc}
\begin{minipage}{0.5\hsize}
\begin{center}
\includegraphics[width=0.8\textwidth]{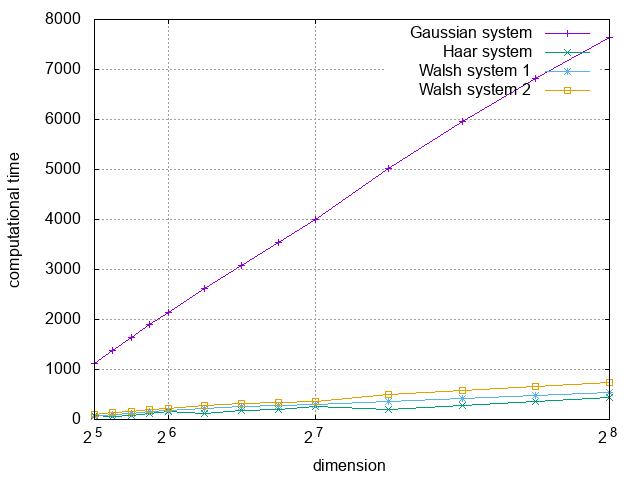}
\caption*{{\tiny (D2) computational time in Case 2 with $n=2^{10}$ and $m=2^{20}$}}
\end{center}
\end{minipage}
\begin{minipage}{0.5\hsize}
\begin{center}
\includegraphics[width=0.8\textwidth]{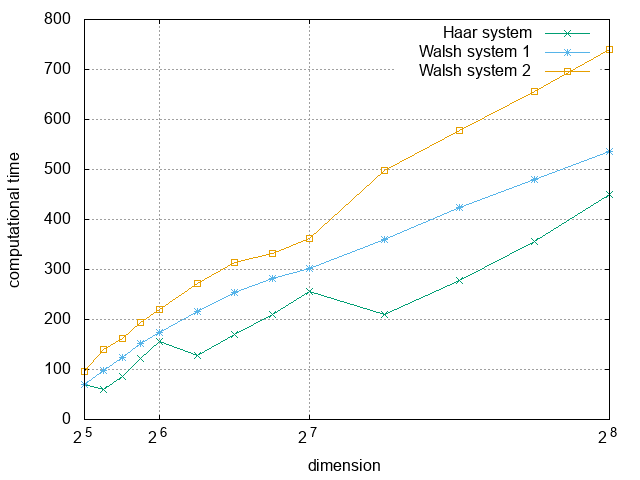}
\caption*{{\tiny (E2) computational time in Case 2 with $n=2^{10}$ and $m=2^{20}$}}
\end{center}
\end{minipage}
\end{tabular}
\begin{tabular}{cc}
\begin{minipage}{0.5\hsize}
\begin{center}
\includegraphics[width=0.8\textwidth]{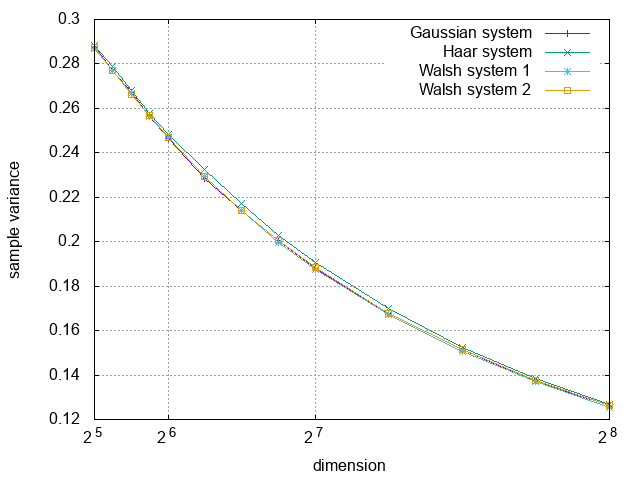}
\caption*{{\tiny (F2) sample variance in Case 2 with $n=2^{10}$ and $m=2^{20}$}}
\end{center}
\end{minipage}
\begin{minipage}{0.5\hsize}
\begin{center}
\includegraphics[width=0.8\textwidth]{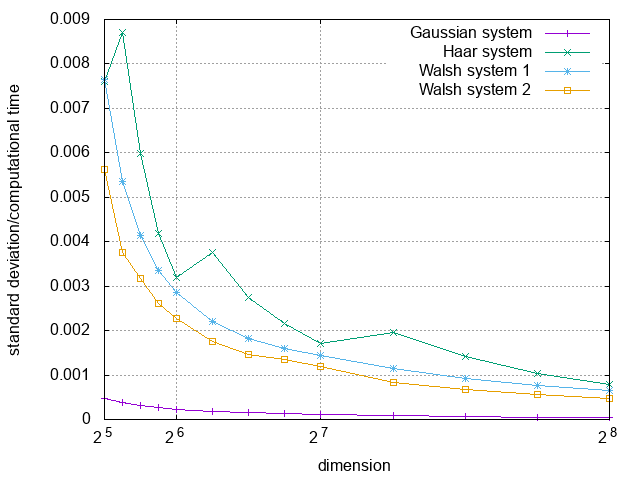}
\caption*{{\tiny (G2) standard deviation/computational time in Case 2 with $n=2^{10}$ and $m=2^{20}$}}
\end{center}
\end{minipage}
\end{tabular}
\end{table}
}

Next, we compare the computational complexity among the Euler-Maruyama schemes $X^{(n)}$ by the Gaussian system, the Haar system \eqref{eq:2.5} and the Walsh system \eqref{eq:2.8} under the following simple situation Case 3.

\begin{itemize}
\item Case 3. We consider a $d$-dimensional Brownian motion $X=W$ with the initial $x_{0} \equiv 0$, the time horizon $T=1$ and the test function $f(x)=|x|^{2}/d$.
Then note that ${\mathbb E}[f(X_{T})]=1$.
\end{itemize}
For a given $\varepsilon>0$, we consider how much time until each EM scheme $X^{(n)}$ satisfies $|1-{\mathbb E}[f(X_{T}^{(n)})]|<\varepsilon$.
These numerical results are shown in the figures (H3)-(K3) below.
The $y$-axis in all the figures indicates the average of the computational time of $100$ trials.
Here, the computational time refers to the time it takes for Monte Calro trials to satisfy $|1-{\mathbb E}[f(X_{T}^{(n)})]|<\varepsilon$ for the first time.
The purple, green, blue and yellow points are results for the EM schemes by the Gaussian system, the Haar system, the Walsh system (named Walsh system 1) with the algorithm created using the odd-ordered map \eqref{odd-ordered map} and the Walsh system (named Walsh system 2) with the simple algorithm created using the bit mask, respectively.
The lines represent the regression lines corresponding to each color.
The figures (I3) and (K3) are enlargements of (H3) and (J3), respectively, excluding the standard EM scheme with the Gaussian system.

In the figures (H3) and (I3), we compare, among the EM schemes, the computational time for each dimension with $n=2^{10}$ and the fixed accuracy $\varepsilon=10^{-3}$.
In the figures (J3) and (K3), we compare, among them, the computational time for each $\varepsilon=5.0 \times10^{-4}, 5.5 \times 10^{-4},\ldots, 9.5 \times 10^{-4}, 1.0 \times 10^{-3}$ with the fixed dimension $d=2^{5}$ and $n=2^{10}$.
From these results with considering the regression lines, we can also see that the Haar system has the highest accuracy, followed the Walsh system by the odd-ordered map, the Walsh system by the bit mask and the Gaussian system.

{\tiny
\begin{table}
\begin{tabular}{cc}
\begin{minipage}{0.5\hsize}
\begin{center}
\includegraphics[width=0.8\textwidth]{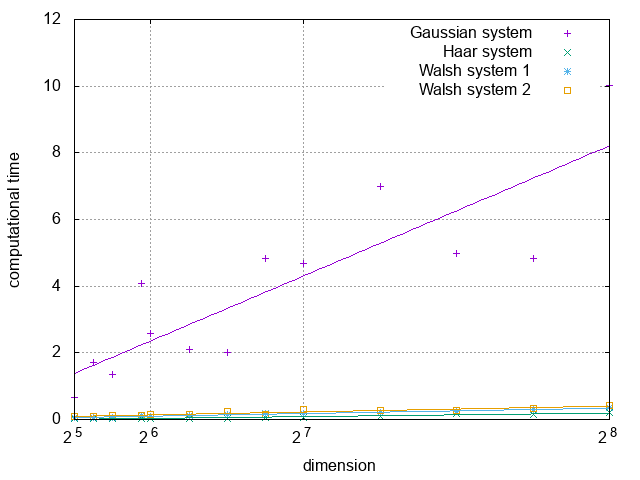}
\caption*{{\tiny (H3) computational complexity for dimension in Case 3 with $n=2^{10}$ and $\varepsilon=10^{-3}$}}
\end{center}
\end{minipage}
\begin{minipage}{0.5\hsize}
\begin{center}
\includegraphics[width=0.8\textwidth]{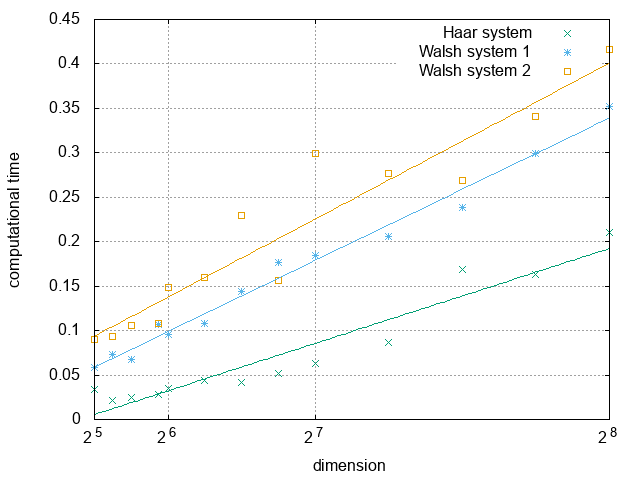}
\caption*{{\tiny (I3) computational complexity for dimension in Case 3 with $n=2^{10}$ and $\varepsilon=10^{-3}$}}
\end{center}
\end{minipage}
\end{tabular}
\begin{tabular}{cc}
\begin{minipage}{0.5\hsize}
\begin{center}
\includegraphics[width=0.8\textwidth]{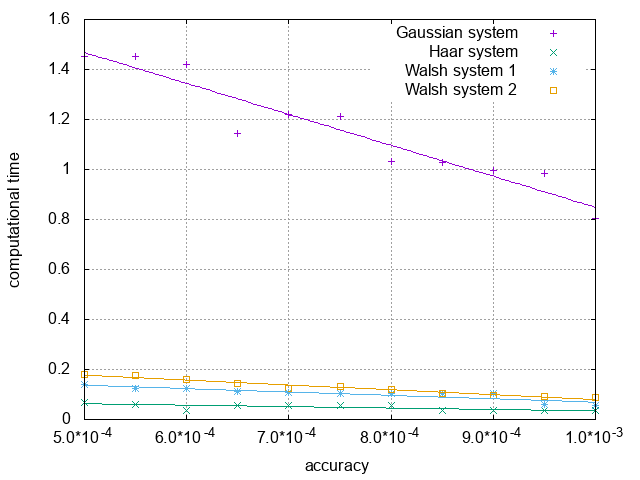}
\caption*{{\tiny (J3) computational complexity for accuracy in Case 3 with $d=2^{5}$ and $n=2^{10}$}}
\end{center}
\end{minipage}
\begin{minipage}{0.5\hsize}
\begin{center}
\includegraphics[width=0.8\textwidth]{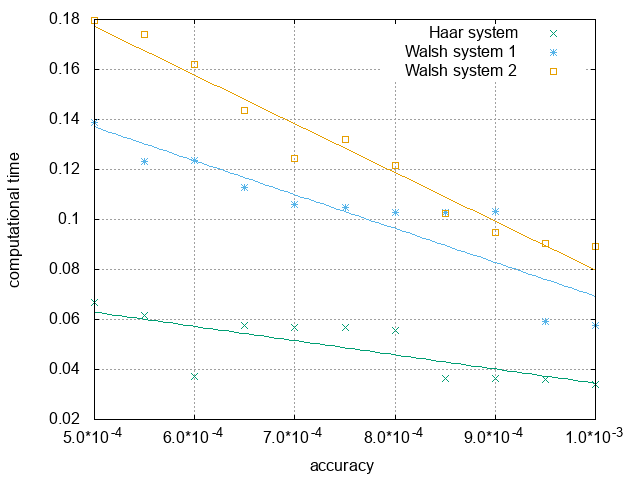}
\caption*{{\tiny (K3) computational complexity for accuracy in Case 3 with $d=2^{5}$ and $n=2^{10}$}}
\end{center}
\end{minipage}
\end{tabular}
\end{table}
}

\section{Appendix: Proof of Theorem \ref{theo:3.3}}
\label{sec:5}
In this section, we prove Theorem \ref{theo:3.3} in the same way as Theorem 14.5.2 in \cite{KEPPE} using the It\^o Taylor expansion (the Wagner-Platen expansion).
We first introduce various notations.

\subsection{Notations in this section}
\label{sec:5.1}
\begin{itemize}
\item The It\^o integral operator is defined as follows: For any $j \in \{0,1,\ldots,d\}$ and $0 \leq s \leq t \leq T$,
$$
I_{s,t}^{(j)}f:=
\left\{
\begin{array}{ll}
\displaystyle{\int_{s}^{t}f(u){\rm d}u} & j=0 \\
\displaystyle{\int_{s}^{t}f(u){\rm d}W_{u}^{j}} & j \in \{1,2,\ldots,d\}
\end{array}
\right., \quad f \in {\rm Dom}(I_{s,t}^{(j)}),
$$
where
\begin{align*}
{\rm Dom}(I_{s,t}^{(j)})&:=
\left\{
\begin{array}{ll}
\left\{f=(f(t))_{t \geq 0}\,;\,\hspace{0.3cm}\text{For any }t>0, \int_{0}^{t}|f(u)|{\rm d}u<\infty.\right\}
 & j=0  \\
\left\{f=(f(t))_{t \geq 0}\,;\,
\begin{array}{ll}
f \text{ is a measurable process on } \Omega \text{ adapted to } ({\mathcal F}_{t})_{t \geq 0} \\
 \text{ such that for any } t>0, \int_{0}^{t}|f(u)|^{2}{\rm d}u<\infty \text{ a.s.}
\end{array}
\right\} & j \in \{1,2,\ldots,d\}.
\end{array}
\right.
\end{align*}
Here, we see that ${\rm Dom}(I_{s,t}^{(j)})$ is the common domain of the It\^o integral operator with respect to time since it does not depend on $s$ and $t$.
In this paper, we only treat the following double It\^o integral operator among the multiple It\^o integral operators:
For any $j_{1}.j_{2} \in \{0,1,\ldots,d\}$ and $0 \leq s \leq t \leq T$,
$$
I_{s,t}^{(j_{1},j_{2})}f:=I_{s,t}^{(j_{2})}I_{s,\bullet}^{(j_{1})}f, \quad f \in \text{Dom}(I_{s,t}^{(j_{1},j_{2})}):=\left\{f=(f(t))_{t \geq 0}\,;\,I_{s,\bullet}^{(j_{1})}f \in {\rm Dom}(I_{s,t}^{(j_{2})})\right\}.
$$
\item The It\^o coefficient function is defined as follows: For any $j \in \{0,1,\ldots,d\}$, $0 \leq s \leq T$ and $y \in {\mathbb R}^{d}$,
$$
{\mathcal L}^{j}f(s,y):=
\left\{
\begin{array}{ll}
\displaystyle{\frac{\partial f(s,y)}{\partial s}+\sum_{i'=1}^{d}b^{i'}(y)\frac{\partial f(s,y)}{\partial y^{i'}}+\frac{1}{2}\sum_{i'_{1},i'_{2},j' \in \{1,2,\ldots,d\}}\sigma_{j'}^{i'_{1}}(y)\sigma_{j'}^{i'_{2}}(y)\frac{\partial^{2}f(s,y)}{\partial y^{i'_{1}}\partial y^{i'_{2}}}} & j=0 \\
\displaystyle{\sum_{i'=1}^{d}\sigma_{j}^{i'}(y)\frac{\partial f(s,y)}{\partial y^{i'}}} & j \in \{1,2,\ldots,d\}
\end{array}
\right., \quad f \in {\rm Dom}({\mathcal L}^{j}),
$$
where
$$
{\rm Dom}({\mathcal L}^{j}):=
\left\{
\begin{array}{ll}
C^{1,2}([0,T] \times {\mathbb R}^{d}) & j=0 \\
C^{0,1}([0,T] \times {\mathbb R}^{d}) & j \in \{1,2,\ldots,d\}.
\end{array}
\right.
$$
\item The flow associated with the SDE is defined as follows: For any $0 \leq s \leq T$ and $y \in {\mathbb R}^{d}$,
$$
X_{t}^{s,y}=y+\sum_{j=1}^{d}\int_{s}^{t}\sigma_{j}(X_{u}^{s,y}){\rm d}W_{u}^{j}+\int_{s}^{t}b(X_{u}^{s,y}){\rm d}u, \quad s \leq t \leq T.
$$
Here, we denote $ \sigma_j $ is the $ j $-th column vector of the matrix $ \sigma $, and its $i$-th component is denoted by $\sigma_{j}^{i}$.
Similarly, the $i$-th component of the vector $b$ is denoted by $b^{i}$.
\item The functional associated with the flow for a fixed test function $f \in C_{P}^{4}({\mathbb R}^{d})$ is defined as follows: For any $0 \leq s \leq T$ and $y \in {\mathbb R}^{d}$,
$$
u(s,y):={\mathbb E}[f(X_{T}^{s,y})].
$$
\item We often use $\int_{t_{\ell-1}}^{t}\bullet\,{\rm d}W_{s}^{0}:=\int_{t_{\ell-1}}^{t}\bullet\,{\rm d}s$ and $\sigma_{0}:=b$ to simplify the argument.
\item We always denote by $ C $ the chosen constants wherever there is no risk of ambiguity.
\end{itemize}

As is well known, we obtain the following two statements under Assumption \ref{ass:3.1}.
Theorem \ref{theo:5.1} states that the flow is a solution of the Kolmogorov backward equation.
The expansion in Theorem \ref{theo:5.2} is called the (first order) It\^o-Taylor expansion or the Wagner-Platen expansion.

\begin{Theorem}[cf. Theorem 4.8.6 in \cite{KEPPE}]
\label{theo:5.1}
Suppose that Assumption \ref{ass:3.1} holds.
Then the functional $u$ associated with the flow satisfies the following two statements:
\begin{itemize}
\item For any $0 \leq s \leq T$, $u(s,\bullet) \in C_{P}^{4}({\mathbb R}^{d})$.
\item For any $0 \leq s \leq T$, $y \in {\mathbb R}^{d}$, ${\mathcal L}^{0}u(s,y)=0$.
\end{itemize}
\end{Theorem}

\begin{Theorem}[cf. Theorem 5.5.1 in \cite{KEPPE}]
\label{theo:5.2}
Suppose that Assumption \ref{ass:3.1} holds.
Then for any $\ell \in \{1,2,\ldots,n\}$, it holds that
$$
X_{t}^{t_{\ell-1},X_{t_{\ell-1}}^{(n)}}=\eta_{\ell}^{(n)}(t)+\sum_{j_{1},j_{2} \in \{0,1,\ldots,d\}}I_{t_{\ell-1},t}^{(j_{1},j_{2})}{\mathcal L}^{j_{1}}\sigma_{j_{2}}(X_{\bullet}^{t_{\ell-1},X_{t_{\ell-1}}^{(n)}}), \quad t_{\ell-1} \leq t \leq t_{\ell},
$$
where 
$$
\eta_{\ell}^{(n)}(t):=X_{t_{\ell-1}}^{(n)}+\sum_{j=0}^{d}I_{t_{\ell-1},t}^{(j)}\sigma_{j}(X_{t_{\ell-1}}^{(n)}), \quad t_{\ell-1} \leq t \leq t_{\ell}.
$$
\end{Theorem}
Next, we give some lemmas that will be used to prove Theorem \ref{theo:3.3}.

\subsection{Lemmas}
\label{sec:5.2}
\begin{Lemma}
\label{lem:5.3}
For any $p \in {\mathbb N}$, there exists a positive constant $C$ such that for any $j \in \{0,1,\ldots,d\}$, $\ell \in \{1,2,\ldots,n\}$ and $f \in \cap_{t \in [t_{\ell-1},t_{\ell}]}{\rm Dom}(I_{t_{\ell-1},t}^{(j)})$,
$$
\sup_{t \in [t_{\ell-1},t_{\ell}]}{\mathbb E}\left[\left|I_{t_{\ell-1},t}^{(j)}f\right|^{2p}\,\Bigr|\,{\mathcal F}_{t_{\ell-1}}\right]
\leq
\left\{
\begin{array}{ll}
\displaystyle{\frac{C}{n^{2p-1}}\int_{t_{\ell-1}}^{t_{\ell}}{\mathbb E}\left[\left|f(s)\right|^{2p}\,\bigr|\,{\mathcal F}_{t_{\ell-1}}\right]{\rm d}s} & j=0 \\
\displaystyle{\frac{C}{n^{p-1}}\int_{t_{\ell-1}}^{t_{\ell}}{\mathbb E}\left[\left|f(s)\right|^{2p}\,\bigr|\,{\mathcal F}_{t_{\ell-1}}\right]{\rm d}s} & j \neq 0.
\end{array}
\right.
$$
\end{Lemma}
\begin{proof}
The statement follows by using Jensen's inequality in $j=0$ and using It\^o's formula in $j \neq 0$.
\end{proof}

\begin{Lemma}
\label{lem:5.4}
Suppose that Assumption \ref{ass:3.1} holds.
Then for any $p \in {\mathbb N}$, there exists a positive constant $C$ such that for any $\ell \in \{1,2,\ldots,n\}$,
$$
\sup_{t \in [t_{\ell-1},t_{\ell}]}{\mathbb E}\left[\Bigl|X_{t}^{t_{\ell-1},X_{t_{\ell-1}}^{(n)}}\Bigr|^{2p}\,\Bigr|\,{\mathcal F}_{t_{\ell-1}}\right]
\leq C\left(1+\left|X_{t_{\ell-1}}^{(n)}\right|^{2p}\right)
$$
and
$$
{\mathbb E}\left[\Bigl|X_{t_{\ell}}^{t_{\ell-1},X_{t_{\ell-1}}^{(n)}}-X_{t_{\ell-1}}^{(n)}\Bigr|^{2p}\,\Bigr|\,{\mathcal F}_{t_{\ell-1}}\right]
\leq \frac{C}{n^{p}}\left(1+\left|X_{t_{\ell-1}}^{(n)}\right|^{2p}\right).
$$
\end{Lemma}
\begin{proof}
Fix $t \in [t_{\ell-1},t_{\ell}]$.
Using It\^o's formula, we obtain
\begin{align*}
\Bigl|X_{t}^{t_{\ell-1},X_{t_{\ell-1}}^{(n)}}\Bigr|^{2p}
&=\left|X_{t_{\ell-1}}^{(n)}\right|^{2p}+2p\sum_{i,j \in \{1,2,\ldots,d\}}\int_{t_{\ell-1}}^{t}\Bigl|X_{s}^{t_{\ell-1},X_{t_{\ell-1}}^{(n)}}\Bigr|^{2(p-1)}(X_{s}^{t_{\ell-1},X_{t_{\ell-1}}^{(n)}})^{i}\sigma_{j}^{i}(X_{s}^{t_{\ell-1},X_{t_{\ell-1}}^{(n)}}){\rm d}W_{s}^{j} \\
&\hspace{0.38cm}+2p\sum_{i=1}^{d}\int_{t_{\ell-1}}^{t}\Bigl|X_{s}^{t_{\ell-1},X_{t_{\ell-1}}^{(n)}}\Bigr|^{2(p-1)}(X_{s}^{t_{\ell-1},X_{t_{\ell-1}}^{(n)}})^{i}b^{i}(X_{s}^{t_{\ell-1},X_{t_{\ell-1}}^{(n)}}){\rm d}s \\
&\hspace{0.38cm}+2p(p-1)\sum_{i,j,k \in \{1,2,\ldots,d\}}\int_{t_{\ell-1}}^{t}\Bigl|X_{s}^{t_{\ell-1},X_{t_{\ell-1}}^{(n)}}\Bigr|^{2(p-2)}(X_{s}^{t_{\ell-1},X_{t_{\ell-1}}^{(n)}})^{i}(X_{s}^{t_{\ell-1},X_{t_{\ell-1}}^{(n)}})^{j}\sigma_{k}^{i}(X_{s}^{t_{\ell-1},X_{t_{\ell-1}}^{(n)}})\sigma_{k}^{j}(X_{s}^{t_{\ell-1},X_{t_{\ell-1}}^{(n)}}){\rm d}s \\
&\hspace{0.38cm}+p\sum_{i,k \in \{1,2,\ldots,d\}}\int_{t_{\ell-1}}^{t}\Bigl|X_{s}^{t_{\ell-1},X_{t_{\ell-1}}^{(n)}}\Bigr|^{2(p-1)}\Bigl|\sigma_{k}^{i}(X_{s}^{t_{\ell-1},X_{t_{\ell-1}}^{(n)}})\Bigr|^{2}{\rm d}s.
\end{align*}
Then by using the Cauchy-Schwarz inequality, {\bf A2-1} and the martingale property, we have
\begin{align*}
{\mathbb E}\left[\Bigl|X_{t}^{t_{\ell-1},X_{t_{\ell-1}}^{(n)}}\Bigr|^{2p}\,\Bigr|\,{\mathcal F}_{t_{\ell-1}}\right]
&\leq \left|X_{t_{\ell-1}}^{(n)}\right|^{2p}
+2p\int_{t_{\ell-1}}^{t}{\mathbb E}\left[\Bigl|X_{s}^{t_{\ell-1},X_{t_{\ell-1}}^{(n)}}\Bigr|^{2p-1}\Bigl|b(X_{s}^{t_{\ell-1},X_{t_{\ell-1}}^{(n)}})\Bigr|\,\Bigr|\,{\mathcal F}_{t_{\ell-1}}\right]{\rm d}s \\
&\hspace{0.38cm}+p(2p-1)\int_{t_{\ell-1}}^{t}{\mathbb E}\left[\Bigl|X_{s}^{t_{\ell-1},X_{t_{\ell-1}}^{(n)}}\Bigr|^{2(p-1)}\Bigl|\sigma(X_{s}^{t_{\ell-1},X_{t_{\ell-1}}^{(n)}})\Bigr|^{2}\,\Bigr|\,{\mathcal F}_{t_{\ell-1}}\right]{\rm d}s.
\end{align*}
Thus by {\bf A1-2}, we obtain
\begin{align*}
{\mathbb E}\left[\Bigl|X_{t}^{t_{\ell-1},X_{t_{\ell-1}}^{(n)}}\Bigr|^{2p}\,\Bigr|\,{\mathcal F}_{t_{\ell-1}}\right]
&\leq \left|X_{t_{\ell-1}}^{(n)}\right|^{2p}+C\int_{t_{\ell-1}}^{t}{\mathbb E}\left[\Bigl|X_{s}^{t_{\ell-1},X_{t_{\ell-1}}^{(n)}}\Bigr|^{2(p-1)}\biggl(1+\Bigl|X_{s}^{t_{\ell-1},X_{t_{\ell-1}}^{(n)}}\Bigr|\biggr)^{2}\,\Bigr|\,{\mathcal F}_{t_{\ell-1}}\right]{\rm d}s \\
&\leq \left|X_{t_{\ell-1}}^{(n)}\right|^{2p}+C\int_{t_{\ell-1}}^{t}\left(1+{\mathbb E}\left[\Bigl|X_{s}^{t_{\ell-1},X_{t_{\ell-1}}^{(n)}}\Bigr|^{2p}\,\Bigr|\,{\mathcal F}_{t_{\ell-1}}\right]\right){\rm d}s \\
&=\left|X_{t_{\ell-1}}^{(n)}\right|^{2p}+C(t-t_{\ell-1})+C\int_{t_{\ell-1}}^{t}{\mathbb E}\left[\Bigl|X_{s}^{t_{\ell-1},X_{t_{\ell-1}}^{(n)}}\Bigr|^{2p}\,\Bigr|\,{\mathcal F}_{t_{\ell-1}}\right]{\rm d}s.
\end{align*}
Hence it follows by using the Gronwall inequality that
\begin{align*}
{\mathbb E}\left[\Bigl|X_{t}^{t_{\ell-1},X_{t_{\ell-1}}^{(n)}}\Bigr|^{2p}\,\Bigr|\,{\mathcal F}_{t_{\ell-1}}\right]
&\leq \left|X_{t_{\ell-1}}^{(n)}\right|^{2p}+C(t-t_{\ell-1})+C\int_{t_{\ell-1}}^{t}e^{C(t-s)}\left(\left|X_{t_{\ell-1}}^{(n)}\right|^{2p}+C(s-t_{\ell-1})\right){\rm d}s \\
&\leq \left|X_{t_{\ell-1}}^{(n)}\right|^{2p}+C(t-t_{\ell-1})+C(t-t_{\ell-1})e^{C(t-t_{\ell-1})}\left(\left|X_{t_{\ell-1}}^{(n)}\right|^{2p}+C(t-t_{\ell-1})\right) \\
&=\left(1+C(t-t_{\ell-1})e^{C(t-t_{\ell-1})}\right)\left(\left|X_{t_{\ell-1}}^{(n)}\right|^{2p}+C(t-t_{\ell-1})\right) \\
&\leq \left(1+CTe^{CT}\right)\left(1 \vee CT\right)\left(1+\left|X_{t_{\ell-1}}^{(n)}\right|^{2p}\right),
\end{align*}
which concludes the first statement.

Next, we first have
\begin{align*}
\left|\sum_{j=0}^{d}\int_{t_{\ell-1}}^{t_{\ell}}\sigma_{j}(X_{s}^{t_{\ell-1},X_{t_{\ell-1}}^{(n)}}){\rm d}W_{s}^{j}\right|^{2p}
&=\left(\sum_{i=1}^{d}\left|\sum_{j=0}^{d}\int_{t_{\ell-1}}^{t_{\ell}}\sigma_{j}^{i}(X_{s}^{t_{\ell-1},X_{t_{\ell-1}}^{(n)}}){\rm d}W_{s}^{j}\right|^{2}\right)^{p} \\
&\leq C\sum_{i=1}^{d}\sum_{j=0}^{d}\left|\int_{t_{\ell-1}}^{t_{\ell}}\sigma_{j}^{i}(X_{s}^{t_{\ell-1},X_{t_{\ell-1}}^{(n)}}){\rm d}W_{s}^{j}\right|^{2p}.
\end{align*}
Then from Lemma \ref{lem:5.3}, {\bf A1-2} and the first statement, we obtain
\begin{align*}
{\mathbb E}\left[\Bigl|X_{t_{\ell}}^{t_{\ell-1},X_{t_{\ell-1}}^{(n)}}-X_{t_{\ell-1}}^{(n)}\Bigr|^{2p}\,\Bigr|\,{\mathcal F}_{t_{\ell-1}}\right]
&\leq \frac{C}{n^{p-1}}\sum_{i=1}^{d}\sum_{j=0}^{d}\int_{t_{\ell-1}}^{t_{\ell}}{\mathbb E}\biggl[\Bigl|\sigma_{j}^{i}(X_{s}^{t_{\ell-1},X_{t_{\ell-1}}^{(n)}})\Bigr|^{2p}\,\Bigr|\,{\mathcal F}_{t_{\ell-1}}\biggr]{\rm d}s \\
&\leq \frac{C}{n^{p-1}}\int_{t_{\ell-1}}^{t_{\ell}}{\mathbb E}\biggl[1+\Bigl|X_{t}^{t_{\ell-1},X_{t_{\ell-1}}^{(n)}}\Bigr|^{2p}\,\Bigr|\,{\mathcal F}_{t_{\ell-1}}\biggr]{\rm d}s \\
&\leq  \frac{C}{n^{p}}\left(1+\left|X_{t_{\ell-1}}^{(n)}\right|^{2p}\right),
\end{align*}
which concludes the second statement.
\end{proof}

\begin{Lemma}
\label{lem:5.5}
Suppose that Assumptions \ref{ass:3.1} and \ref{ass:3.2} hold.
Then there exists a positive constant $C$ such that for any $k \in \{1,2,3\}$, $i_{1},i_{2},\ldots,i_{k} \in \{1,2,\ldots,d\}$ and $\ell \in \{1,2,\ldots,n\}$,
\begin{align}
\label{eq:5.1}
\left|{\mathbb E}\left[\prod_{j=1}^{k}(X_{t_{\ell}}^{t_{\ell-1},X_{t_{\ell-1}}^{(n)}}-X_{t_{\ell-1}}^{(n)})^{i_{j}}-\prod_{j=1}^{k}(\eta_{\ell}^{(n)}(t_{\ell})-X_{t_{\ell-1}}^{(n)})^{i_{j}}\,\Bigggr|\,{\mathcal F}_{t_{\ell-1}}\right]\right| \leq \frac{C}{n^{2}}\left(1+\left|X_{t_{\ell-1}}^{(n)}\right|^{6(r+1)}\right),
\end{align}
where
$$
r:=\min\left\{r \in {\mathbb N}\cup \{0\} \,;\, 
\begin{array}{ll}
\text{There exists a positive constant } C \text{ such that for any } y \in {\mathbb R}^{d}, \\
\displaystyle{\max_{\substack{i \in \{1,2,\ldots,d\} \\ j \in \{0,1,\ldots,d\}}}\max_{\substack{i_{1},i_{2},\ldots,i_{k} \in \{1,2,\ldots,d\} \\ k \in \{1,2\}}}\left|\frac{\partial^{k}\sigma_{j}^{i}(y)}{\partial y^{i_{1}}\partial y^{i_{2}}\cdots\partial y^{i_{k}}}\right| \leq C\left(1+|y|^{2r}\right)}.
\end{array}
\right\}.
$$
\end{Lemma}
\begin{proof}
\begin{itemize}
\item First, we show \eqref{eq:5.1} for $k=1$.
From Theorem \ref{theo:5.2} and the martingale property, we obtain
\begin{align*}
&{\mathbb E}\left[(X_{t_{\ell}}^{t_{\ell-1},X_{t_{\ell-1}}^{(n)}}-X_{t_{\ell-1}}^{(n)})^{i_{1}}-(\eta_{\ell}^{(n)}(t_{\ell})-X_{t_{\ell-1}}^{(n)})^{i_{1}}\,\Bigr|\,{\mathcal F}_{t_{\ell-1}}\right] \\
&={\mathbb E}\left[(X_{t_{\ell}}^{t_{\ell-1},X_{t_{\ell-1}}^{(n)}}-\eta_{\ell}^{(n)}(t_{\ell}))^{i_{1}}\,\Bigr|\,{\mathcal F}_{t_{\ell-1}}\right]
=\sum_{j_{1},j_{2} \in \{0,1,\ldots,d\}}{\mathbb E}\left[I_{t_{\ell-1},t_{\ell}}^{(j_{1},j_{2})}{\mathcal L}^{j_{1}}\sigma_{j_{2}}^{i_{1}}(X_{\bullet}^{t_{\ell-1},X_{t_{\ell-1}}^{(n)}})\,\Bigr|\,{\mathcal F}_{t_{\ell-1}}\right] \\
&={\mathbb E}\left[I_{t_{\ell-1},t_{\ell}}^{(0,0)}{\mathcal L}^{0}b^{i_{1}}(X_{\bullet}^{t_{\ell-1},X_{t_{\ell-1}}^{(n)}})\,\Bigr|\,{\mathcal F}_{t_{\ell-1}}\right]
=\int_{t_{\ell-1}}^{t_{\ell}}\int_{t_{\ell-1}}^{s_{2}}{\mathbb E}\left[{\mathcal L}^{0}b^{i_{1}}(X_{s_{1}}^{t_{\ell-1},X_{t_{\ell-1}}^{(n)}})\,\Bigr|\,{\mathcal F}_{t_{\ell-1}}\right]{\rm d}s_{1}{\rm d}s_{2}.
\end{align*}
On the other hand, by using the Cauchy-Schwarz inequality, {\bf A1-2} and {\bf A1-3}, we have for any $i \in \{1,2,\ldots,d\}$, $j_{1},j_{2} \in \{0,1,\ldots,d\}$ and $y \in {\mathbb R}^{d}$,
\begin{align}
\label{eq:5.2}
\left|{\mathcal L}^{j_{1}}\sigma_{j_{2}}^{i}(y)\right|^{2}
&\leq C\left(\left|\sum_{i'=1}^{d}\sigma_{j_{1}}^{i'}(y)\frac{\partial \sigma_{j_{2}}^{i}(y)}{\partial y^{i'}}\right|^{2}+\left|\sum_{i'_{1},i'_{2},j' \in \{1,2,\ldots,d\}}\sigma_{j'}^{i'_{1}}(y)\sigma_{j'}^{i'_{2}}(y)\frac{\partial \sigma_{j_{2}}^{i}(y)}{\partial y^{i'_{1}}\partial y^{i'_{2}}}\right|^{2}\right) \\ \notag
&\leq C\left(|\sigma_{j_{1}}(y)|^{2}\sum_{i'=1}^{d}\left|\frac{\partial\sigma_{j_{2}}^{i}(y)}{\partial y^{i'}}\right|^{2}+|\sigma(y)|^{4}\sum_{i'_{1},i'_{2}=1}^{d}\left|\frac{\partial^{2}\sigma_{j_{2}}^{i}(y)}{\partial y^{i'_{1}}\partial y^{i'_{2}}}\right|^{2}\right) \\ \notag
&\leq C\left(1+|y|^{4(r+1)}\right).
\end{align}
Thus it follows from Lemma \ref{lem:5.4} that
$$
\left|{\mathbb E}\left[(X_{t_{\ell}}^{t_{\ell-1},X_{t_{\ell-1}}^{(n)}}-X_{t_{\ell-1}}^{(n)})^{i_{1}}-(\eta_{\ell}^{(n)}(t_{\ell})-X_{t_{\ell-1}}^{(n)})^{i_{1}}\,\Bigr|\,{\mathcal F}_{t_{\ell-1}}\right]\right|
\leq \frac{C}{n^{2}}\left(1+\left|X_{t_{\ell-1}}^{(n)}\right|^{2(r+1)}\right).
$$
\item Next, we show \eqref{eq:5.1} for $k=2$.
We first obtain
\begin{align}
\label{eq:5.3}
&{\mathbb E}\left[\prod_{j=1}^{2}(X_{t_{\ell}}^{t_{\ell-1},X_{t_{\ell-1}}^{(n)}}-X_{t_{\ell-1}}^{(n)})^{i_{j}}-\prod_{j=1}^{2}(\eta_{\ell}^{(n)}(t_{\ell})-X_{t_{\ell-1}}^{(n)})^{i_{j}}\,\Bigggr|\,{\mathcal F}_{t_{\ell-1}}\right] \\ \notag
&={\mathbb E}\left[(\eta_{\ell}^{(n)}(t_{\ell})-X_{t_{\ell-1}}^{(n)})^{i_{1}}(X_{t_{\ell}}^{t_{\ell-1},X_{t_{\ell-1}}^{(n)}}-\eta_{\ell}^{(n)}(t_{\ell}))^{i_{2}}\,\Bigr|\,{\mathcal F}_{t_{\ell-1}}\right] \\ \notag
&\hspace{0.38cm}+{\mathbb E}\left[(X_{t_{\ell}}^{t_{\ell-1},X_{t_{\ell-1}}^{(n)}}-\eta_{\ell}^{(n)}(t_{\ell}))^{i_{1}}(\eta_{\ell}^{(n)}(t_{\ell})-X_{t_{\ell-1}}^{(n)})^{i_{2}}\,\Bigr|\,{\mathcal F}_{t_{\ell-1}}\right] \\ \notag
&\hspace{0.38cm}+{\mathbb E}\Biggl[(X_{t_{\ell}}^{t_{\ell-1},X_{t_{\ell-1}}^{(n)}}-\eta_{\ell}^{(n)}(t_{\ell}))^{i_{1}}(X_{t_{\ell}}^{t_{\ell-1},X_{t_{\ell-1}}^{(n)}}-\eta_{\ell}^{(n)}(t_{\ell}))^{i_{2}}\,\Bigr|\,{\mathcal F}_{t_{\ell-1}}\Biggr].
\end{align}
Then we estimate each term on the right hand side.
From Theorem \ref{theo:5.2}, the martingale property, some properties of the stochastic integral (cf. Proposition 2-1.1, 2-1.2 in \cite{INWS}) and {\bf A2-1}, we have
\begin{align*}
&{\mathbb E}\left[(\eta_{\ell}^{(n)}(t_{\ell})-X_{t_{\ell-1}}^{(n)})^{i_{1}}(X_{t_{\ell}}^{t_{\ell-1},X_{t_{\ell-1}}^{(n)}}-\eta_{\ell}^{(n)}(t_{\ell}))^{i_{2}}\,\Bigr|\,{\mathcal F}_{t_{\ell-1}}\right] \\
&=\sum_{j,j_{1},j_{2} \in \{0,1,\ldots,d\}}{\mathbb E}\left[I_{t_{\ell-1},t_{\ell}}^{(j)}\sigma_{j}^{i_{1}}(X_{t_{\ell-1}}^{(n)}) \cdot I_{t_{\ell-1},t_{\ell}}^{(j_{1},j_{2})}{\mathcal L}^{j_{1}}\sigma_{j_{2}}^{i_{2}}(X_{\bullet}^{t_{\ell-1},X_{t_{\ell-1}}^{(n)}})\,\Bigr|\,{\mathcal F}_{t_{\ell-1}}\right] \\
&=\sum_{j=0}^{d}{\mathbb E}\left[I_{t_{\ell-1},t_{\ell}}^{(j)}\sigma_{j}^{i_{1}}(X_{t_{\ell-1}}^{(n)}) \cdot I_{t_{\ell-1},t_{\ell}}^{(0,0)}{\mathcal L}^{0}b^{i_{2}}(X_{\bullet}^{t_{\ell-1},X_{t_{\ell-1}}^{(n)}})\,\Bigr|\,{\mathcal F}_{t_{\ell-1}}\right] \\
&\hspace{0.38cm}+\sum_{j=1}^{d}{\mathbb E}\left[I_{t_{\ell-1},t_{\ell}}^{(j)}\sigma_{j}^{i_{1}}(X_{t_{\ell-1}}^{(n)}) \cdot I_{t_{\ell-1},t_{\ell}}^{(j,0)}{\mathcal L}^{j}b^{i_{2}}(X_{\bullet}^{t_{\ell-1},X_{t_{\ell-1}}^{(n)}})\,\Bigr|\,{\mathcal F}_{t_{\ell-1}}\right] \\
&\hspace{0.38cm}+\sum_{j=1}^{d}{\mathbb E}\left[I_{t_{\ell-1},t_{\ell}}^{(j)}\sigma_{j}^{i_{1}}(X_{t_{\ell-1}}^{(n)}) \cdot I_{t_{\ell-1},t_{\ell}}^{(0,j)}{\mathcal L}^{0}\sigma_{j}^{i_{2}}(X_{\bullet}^{t_{\ell-1},X_{t_{\ell-1}}^{(n)}})\,\Bigr|\,{\mathcal F}_{t_{\ell-1}}\right] \\
&=\sum_{j=0}^{d}\int_{t_{\ell-1}}^{t_{\ell}}\int_{t_{\ell-1}}^{s_{2}}{\mathbb E}\left[I_{t_{\ell-1},t_{\ell}}^{(j)}\sigma_{j}^{i_{1}}(X_{t_{\ell-1}}^{(n)}) \cdot {\mathcal L}^{0}b^{i_{2}}(X_{s_{1}}^{t_{\ell-1},X_{t_{\ell-1}}^{(n)}})\,\Bigr|\,{\mathcal F}_{t_{\ell-1}}\right]{\rm d}s_{1}{\rm d}s_{2} \\
&\hspace{0.38cm}+\sum_{j=1}^{d}\int_{t_{\ell-1}}^{t_{\ell}}\int_{t_{\ell-1}}^{s_{2}}\sigma_{j}^{i_{1}}(X_{t_{\ell-1}}^{(n)}){\mathbb E}\left[{\mathcal L}^{j}b^{i_{2}}(X_{s_{1}}^{t_{\ell-1},X_{t_{\ell-1}}^{(n)}})\,\Bigr|\,{\mathcal F}_{t_{\ell-1}}\right]{\rm d}s_{1}{\rm d}s_{2} \\
&\hspace{0.38cm}+\sum_{j=1}^{d}\int_{t_{\ell-1}}^{t_{\ell}}\int_{t_{\ell-1}}^{s_{2}}\sigma_{j}^{i_{1}}(X_{t_{\ell-1}}^{(n)}){\mathbb E}\left[{\mathcal L}^{0}\sigma_{j}^{i_{2}}(X_{s_{1}}^{t_{\ell-1},X_{t_{\ell-1}}^{(n)}})\,\Bigr|\,{\mathcal F}_{t_{\ell-1}}\right]{\rm d}s_{1}{\rm d}s_{2}.
\end{align*}
Thus by using the Cauchy-Schwarz inequality, {\bf A2-1}, {\bf A1-2}, \eqref{eq:5.2} and Lemma \ref{lem:5.4}, we obtain
\begin{align}
\label{eq:5.4}
\left|{\mathbb E}\left[(\eta_{\ell}^{(n)}(t_{\ell})-X_{t_{\ell}}^{(n)})^{i_{1}}(X_{t_{\ell}}^{t_{\ell-1},X_{t_{\ell-1}}^{(n)}}-\eta_{\ell}^{(n)}(t_{\ell}))^{i_{2}}\,\Bigr|\,{\mathcal F}_{t_{\ell-1}}\right]\right|
\leq \frac{C}{n^{2}}\left(1+\left|X_{t_{\ell-1}}^{(n)}\right|^{2r+3}\right).
\end{align}
On the other hand, from Theorem \ref{theo:5.2}, the Cauchy-Schwarz inequality, Lemma \ref{lem:5.3}, \eqref{eq:5.2} and Lemma \ref{lem:5.4}, we have
\begin{align}
\label{eq:5.5}
&\left|{\mathbb E}\Biggl[(X_{t_{\ell}}^{t_{\ell-1},X_{t_{\ell-1}}^{(n)}}-\eta_{\ell}^{(n)}(t_{\ell}))^{i_{1}}(X_{t_{\ell}}^{t_{\ell-1},X_{t_{\ell-1}}^{(n)}}-\eta_{\ell}^{(n)}(t_{\ell}))^{i_{2}}\,\Bigr|\,{\mathcal F}_{t_{\ell-1}}\Biggr]\right| \\ \notag
&\leq \sum_{j_{1},j_{2},j'_{1},j'_{2} \in \{0,1,\ldots,d\}}{\mathbb E}\left[\Bigl|I_{t_{\ell-1},t_{\ell}}^{(j_{1},j_{2})}{\mathcal L}^{j_{1}}\sigma_{j_{2}}^{i_{1}}(X_{\bullet}^{t_{\ell-1},X_{t_{\ell-1}}^{(n)}})\Bigr|^{2}\,\Bigr|\,{\mathcal F}_{t_{\ell-1}}\right]^{1/2}{\mathbb E}\left[\Bigl|I_{t_{\ell-1},t_{\ell}}^{(j'_{1},j'_{2})}{\mathcal L}^{j'_{1}}\sigma_{j'_{2}}^{i_{1}}(X_{\bullet}^{t_{\ell-1},X_{t_{\ell-1}}^{(n)}})\Bigr|^{2}\,\Bigr|\,{\mathcal F}_{t_{\ell-1}}\right]^{1/2} \\ \notag
&\leq \frac{C}{n^{2}}\left(1+\left|X_{t_{\ell-1}}^{(n)}\right|^{4(r+1)}\right).
\end{align}
Hence it follows by \eqref{eq:5.3}-\eqref{eq:5.5} that
\begin{align*}
{\mathbb E}\left[\prod_{j=1}^{2}(X_{t_{\ell}}^{t_{\ell-1},X_{t_{\ell-1}}^{(n)}}-X_{t_{\ell-1}}^{(n)})^{i_{j}}-\prod_{j=1}^{2}(\eta_{\ell}^{(n)}(t_{\ell})-X_{t_{\ell-1}}^{(n)})^{i_{j}}\,\Bigggr|\,{\mathcal F}_{t_{\ell-1}}\right]
\leq \frac{C}{n^{2}}\left(1+\left|X_{t_{\ell-1}}^{(n)}\right|^{4(r+1)}\right).
\end{align*}
\item Finally, we show \eqref{eq:5.1} for $k=3$.
We first obtain
\begin{align}
\label{eq:5.6}
&{\mathbb E}\left[\prod_{j=1}^{3}(X_{t_{\ell}}^{t_{\ell-1},X_{t_{\ell-1}}^{(n)}}-X_{t_{\ell-1}}^{(n)})^{i_{j}}-\prod_{j=1}^{3}(\eta_{\ell}^{(n)}(t_{\ell})-X_{t_{\ell-1}}^{(n)})^{i_{j}}\,\Bigggr|\,{\mathcal F}_{t_{\ell-1}}\right] \\ \notag
&={\mathbb E}\left[(\eta_{\ell}^{(n)}(t_{\ell})-X_{t_{\ell-1}}^{(n)})^{i_{1}}(\eta_{\ell}^{(n)}(t_{\ell})-X_{t_{\ell-1}}^{(n)})^{i_{2}}(X_{t_{\ell}}^{t_{\ell-1},X_{t_{\ell-1}}^{(n)}}-\eta_{\ell}^{(n)}(t_{\ell}))^{i_{3}}\,\Bigr|\,{\mathcal F}_{t_{\ell-1}}\right] \\ \notag
&\hspace{0.38cm}+{\mathbb E}\left[(\eta_{\ell}^{(n)}(t_{\ell})-X_{t_{\ell-1}}^{(n)})^{i_{1}}(X_{t_{\ell}}^{t_{\ell-1},X_{t_{\ell-1}}^{(n)}}-\eta_{\ell}^{(n)}(t_{\ell}))^{i_{2}}(\eta_{\ell}^{(n)}(t_{\ell})-X_{t_{\ell-1}}^{(n)})^{i_{3}}\,\Bigr|\,{\mathcal F}_{t_{\ell-1}}\right] \\ \notag
&\hspace{0.38cm}+{\mathbb E}\left[(X_{t_{\ell}}^{t_{\ell-1},X_{t_{\ell-1}}^{(n)}}-\eta_{\ell}^{(n)}(t_{\ell}))^{i_{1}}(\eta_{\ell}^{(n)}(t_{\ell})-X_{t_{\ell-1}}^{(n)})^{i_{2}}(\eta_{\ell}^{(n)}(t_{\ell})-X_{t_{\ell-1}}^{(n)})^{i_{3}}\,\Bigr|\,{\mathcal F}_{t_{\ell-1}}\right] \\ \notag
&\hspace{0.38cm}+{\mathbb E}\left[(\eta_{\ell}^{(n)}(t_{\ell})-X_{t_{\ell-1}}^{(n)})^{i_{1}}(X_{t_{\ell}}^{t_{\ell-1},X_{t_{\ell-1}}^{(n)}}-\eta_{\ell}^{(n)}(t_{\ell}))^{i_{2}}(X_{t_{\ell}}^{t_{\ell-1},X_{t_{\ell-1}}^{(n)}}-\eta_{\ell}^{(n)}(t_{\ell}))^{i_{3}}\,\Bigr|\,{\mathcal F}_{t_{\ell-1}}\right] \\ \notag
&\hspace{0.38cm}+{\mathbb E}\left[(X_{t_{\ell}}^{t_{\ell-1},X_{t_{\ell-1}}^{(n)}}-\eta_{\ell}^{(n)}(t_{\ell}))^{i_{1}}(\eta_{\ell}^{(n)}(t_{\ell})-X_{t_{\ell-1}}^{(n)})^{i_{2}}(X_{t_{\ell}}^{t_{\ell-1},X_{t_{\ell-1}}^{(n)}}-\eta_{\ell}^{(n)}(t_{\ell}))^{i_{3}}\,\Bigr|\,{\mathcal F}_{t_{\ell-1}}\right] \\ \notag
&\hspace{0.38cm}+{\mathbb E}\left[(X_{t_{\ell}}^{t_{\ell-1},X_{t_{\ell-1}}^{(n)}}-\eta_{\ell}^{(n)}(t_{\ell}))^{i_{1}}(X_{t_{\ell}}^{t_{\ell-1},X_{t_{\ell-1}}^{(n)}}-\eta_{\ell}^{(n)}(t_{\ell}))^{i_{2}}(\eta_{\ell}^{(n)}(t_{\ell})-X_{t_{\ell-1}}^{(n)})^{i_{3}}\,\Bigr|\,{\mathcal F}_{t_{\ell-1}}\right] \\ \notag
&\hspace{0.38cm}+{\mathbb E}\left[(X_{t_{\ell}}^{t_{\ell-1},X_{t_{\ell-1}}^{(n)}}-\eta_{\ell}^{(n)}(t_{\ell}))^{i_{1}}(X_{t_{\ell}}^{t_{\ell-1},X_{t_{\ell-1}}^{(n)}}-\eta_{\ell}^{(n)}(t_{\ell}))^{i_{2}}(X_{t_{\ell}}^{t_{\ell-1},X_{t_{\ell-1}}^{(n)}}-\eta_{\ell}^{(n)}(t_{\ell}))^{i_{3}}\,\Bigr|\,{\mathcal F}_{t_{\ell-1}}\right].
\end{align}
Then we estimate each term on the right hand side.
From Theorem \ref{theo:5.2}, the Cauchy-Schwarz inequality, Lemma \ref{lem:5.3}, {\bf A2-1}, {\bf A1-2}, \eqref{eq:5.2} and Lemma \ref{lem:5.4}, we have
\begin{align}
\label{eq:5.7}
&\left|{\mathbb E}\left[(\eta_{\ell}^{(n)}(t_{\ell})-X_{t_{\ell-1}}^{(n)})^{i_{1}}(\eta_{\ell}^{(n)}(t_{\ell})-X_{t_{\ell-1}}^{(n)})^{i_{2}}(X_{t_{\ell}}^{t_{\ell-1},X_{t_{\ell-1}}^{(n)}}-\eta_{\ell}^{(n)}(t_{\ell}))^{i_{3}}\,\Bigr|\,{\mathcal F}_{t_{\ell-1}}\right]\right| \\ \notag
&\leq \sum_{j,j',j_{1},j_{2} \in \{0,1,\ldots,d\}}{\mathbb E}\left[\Bigl|I_{t_{\ell-1},t_{\ell}}^{(j)}\sigma_{j}^{i_{1}}(X_{t_{\ell-1}}^{(n)})\Bigr|^{4}\,\Bigr|\,{\mathcal F}_{t_{\ell-1}}\right]^{1/4}{\mathbb E}\left[\Bigl|I_{t_{\ell-1},t_{\ell}}^{(j')}\sigma_{j'}^{i_{1}}(X_{t_{\ell-1}}^{(n)})\Bigr|^{4}\,\Bigr|\,{\mathcal F}_{t_{\ell-1}}\right]^{1/4} \\ \notag
&\hspace{2.6cm}\times {\mathbb E}\left[\Bigl|I_{t_{\ell-1},t_{\ell}}^{(j_{1},j_{2})}{\mathcal L}^{j_{1}}\sigma_{j_{2}}^{i_{3}}(X_{\bullet}^{t_{\ell-1},X_{t_{\ell-1}}^{(n)}})\Bigr|^{2}\,\Bigr|\,{\mathcal F}_{t_{\ell-1}}\right]^{1/2} \\ \notag
&\leq \frac{C}{n^{2}}\left(1+\left|X_{t_{\ell-1}}^{(n)}\right|^{2(r+2)}\right).
\end{align}
Similarly, we obtain
\begin{align}
\label{eq:5.8}
{\mathbb E}\left[(\eta_{\ell}^{(n)}(t_{\ell})-X_{t_{\ell-1}}^{(n)})^{i_{1}}(X_{t_{\ell}}^{t_{\ell-1},X_{t_{\ell-1}}^{(n)}}-\eta_{\ell}^{(n)}(t_{\ell}))^{i_{2}}(X_{t_{\ell}}^{t_{\ell-1},X_{t_{\ell-1}}^{(n)}}-\eta_{\ell}^{(n)}(t_{\ell}))^{i_{3}}\,\Bigr|\,{\mathcal F}_{t_{\ell-1}}\right]
&\leq \frac{C}{n^{5/2}}\left(1+\left|X_{t_{\ell-1}}^{(n)}\right|^{4r+5}\right), \\ \notag
{\mathbb E}\left[(X_{t_{\ell}}^{t_{\ell-1},X_{t_{\ell-1}}^{(n)}}-\eta_{\ell}^{(n)}(t_{\ell}))^{i_{1}}(X_{t_{\ell}}^{t_{\ell-1},X_{t_{\ell-1}}^{(n)}}-\eta_{\ell}^{(n)}(t_{\ell}))^{i_{2}}(X_{t_{\ell}}^{t_{\ell-1},X_{t_{\ell-1}}^{(n)}}-\eta_{\ell}^{(n)}(t_{\ell}))^{i_{3}}\,\Bigr|\,{\mathcal F}_{t_{\ell-1}}\right]
&\leq \frac{C}{n^{3}}\left(1+\left|X_{t_{\ell-1}}^{(n)}\right|^{6(r+1)}\right).
\end{align}
Thus it follows by combining \eqref{eq:5.6}-\eqref{eq:5.8} that
\begin{align*}
{\mathbb E}\left[\prod_{j=1}^{3}(X_{t_{\ell}}^{t_{\ell-1},X_{t_{\ell-1}}^{(n)}}-X_{t_{\ell-1}}^{(n)})^{i_{j}}-\prod_{j=1}^{3}(\eta_{\ell}^{(n)}(t_{\ell})-X_{t_{\ell-1}}^{(n)})^{i_{j}}\,\Bigggr|\,{\mathcal F}_{t_{\ell-1}}\right] \leq \frac{C}{n^{2}}\left(1+\left|X_{t_{\ell-1}}^{(n)}\right|^{6(r+1)}\right).
\end{align*}
\end{itemize}
Therefore, we conclude the statement.
\end{proof}

\begin{Lemma}
\label{lem:5.6}
Suppose that Assumption \ref{ass:3.1} holds.
Then for any $\ell \in \{1,2,\ldots,n\}$, it holds that
$$
{\mathbb E}\left[u(t_{\ell},X_{t_{\ell}}^{t_{\ell-1},X_{t_{\ell-1}}^{(n)}})-u(t_{\ell-1},X_{t_{\ell-1}}^{(n)})\,\biggr|\,{\mathcal F}_{t_{\ell-1}}\right]=0.
$$
\end{Lemma}
\begin{proof}
By using It\^o's formula and Theorem \ref{theo:5.1} we obtain
$$
u(t_{\ell},X_{t_{\ell}}^{t_{\ell-1},X_{t_{\ell-1}}^{(n)}})-u(t_{\ell-1},X_{t_{\ell-1}}^{(n)})=\sum_{j=0}^{d}\int_{t_{\ell-1}}^{t_{\ell}}\left.{\mathcal L}^{j}u(s,\bullet)\right|_{\bullet=X_{s}^{t_{\ell-1},X_{t_{\ell-1}}^{(n)}}}{\rm d}W_{s}^{j}.
$$
Thus the statement follows by the martingale property and Theorem \ref{theo:5.1}.
\end{proof}

\begin{Lemma}
\label{lem:5.7}
Suppose that Assumptions \ref{ass:3.1} and \ref{ass:3.2} hold.
Then for any $p \in {\mathbb N}$, there exists a positive constant $C$ such that for any $\ell \in \{1,2,\ldots,n\}$,
$$
{\mathbb E}\left[\left|X_{t_{\ell}}^{(n)}-X_{t_{\ell-1}}^{(n)}\right|^{2p}\,\Bigr|\,{\mathcal F}_{t_{\ell-1}}\right]
\leq C\left({\mathbb E}\left[\left|\Delta Z_{\ell}^{(n)}\right|^{2p}\right]+\frac{1}{n^{2p}}\right)\left(1+\left|X_{t_{\ell-1}}^{(n)}\right|^{2p}\right)
$$
and
$$
{\mathbb E}\left[\left|X_{t_{\ell}}^{(n)}\right|^{2p}\,\Bigr|\,{\mathcal F}_{0}\right]
\leq C\exp\left\{C\left(1+n\left(M_{2}^{(n)}(Z) \vee M_{2p}^{(n)}(Z)\right)\right)\right\}.
$$
\end{Lemma}
\begin{proof}
By using the Cauchy-Schwarz inequality, we obtain
$$
\left|\sum_{j=1}^{d}\sigma_{j}(X_{t_{\ell-1}}^{(n)})(\Delta Z_{\ell}^{(n)})^{j}\right|^{2}
=\sum_{i=1}^{d}\left|\sum_{j=1}^{d}\sigma_{j}^{i}(X_{t_{\ell-1}}^{(n)})(\Delta Z_{\ell}^{(n)})^{j}\right|^{2}
\leq \left|\sigma(X_{t_{\ell-1}}^{(n)})\right|^{2}\left|\Delta Z_{\ell}^{(n)}\right|^{2}.
$$
Thus by {\bf A1-2}, we have
\begin{align}
\label{eq:5.9}
\left|X_{t_{\ell}}^{(n)}-X_{t_{\ell-1}}^{(n)}\right|^{2p}
&\leq C\left(\left|\sigma(X_{t_{\ell-1}}^{(n)})\right|^{2p}\left|\Delta Z_{\ell}^{(n)}\right|^{2p}+\left|b(X_{t_{\ell-1}}^{(n)})\right|^{2p}\frac{1}{n^{2p}}\right) \\ \notag
&\leq C\left(\left|\Delta Z_{\ell}^{(n)}\right|^{2p}+\frac{1}{n^{2p}}\right)\left(1+\left|X_{t_{\ell-1}}^{(n)}\right|^{2p}\right).
\end{align}
Hence by {\bf A2-1} and {\bf A2-2}, we obtain
$$
{\mathbb E}\left[\left|X_{t_{\ell}}^{(n)}-X_{t_{\ell-1}}^{(n)}\right|^{2p}\,\Bigr|\,{\mathcal F}_{t_{\ell-1}}\right]
\leq C\left({\mathbb E}\left[\left|\Delta Z_{\ell}^{(n)}\right|^{2p}\right]+\frac{1}{n^{2p}}\right)\left(1+\left|X_{t_{\ell-1}}^{(n)}\right|^{2p}\right),
$$
which concludes the first statement.

Next, fix $k \in \{1,2,\ldots,\ell\}$.
By using Taylor's theorem, there exists a $d \times d$ diagonal matrix $\theta_{k}$ such that for any $i \in \{1,2,\ldots,d\}$, $(\theta_{k})_{i}^{i} \in (0,1)$ and
\begin{align*}
\left|X_{t_{k}}^{(n)}\right|^{2p}-\left|X_{t_{k-1}}^{(n)}\right|^{2p}
&=2p\left|X_{t_{k-1}}^{(n)}\right|^{2(p-1)}\sum_{i=1}^{d}(X_{t_{k}}^{(n)}-X_{t_{k-1}}^{(n)})^{i}(X_{t_{k-1}}^{(n)})^{i} \\
&\hspace{0.38cm}+2p(p-1)\left|X_{t_{k-1}}^{(n)}+\theta(X_{t_{k}}^{(n)}-X_{t_{k-1}}^{(n)})\right|^{2(p-2)}\left|\sum_{i=1}^{d}(X_{t_{k}}^{(n)}-X_{t_{k-1}}^{(n)})^{i}(X_{t_{k-1}}^{(n)}+\theta(X_{t_{k}}^{(n)}-X_{t_{k-1}}^{(n)}))^{i}\right|^{2} \\
&\hspace{0.38cm}+p\left|X_{t_{k-1}}^{(n)}+\theta(X_{t_{k}}^{(n)}-X_{t_{k-1}}^{(n)})\right|^{2(p-1)}\left|X_{t_{k}}^{(n)}-X_{t_{k-1}}^{(n)}\right|^{2}.
\end{align*}
Then by {\bf A2-1}-{\bf A2-3}, the Cauchy-Schwarz inequality, {\bf A1-2} and \eqref{eq:5.9}, we have
\begin{align*}
{\mathbb E}\left[\left|X_{t_{k}}^{(n)}\right|^{2p}\,\Bigr|\,{\mathcal F}_{t_{k-1}}\right]-\left|X_{t_{k-1}}^{(n)}\right|^{2p}
&\leq \frac{C}{n}\left|X_{t_{k-1}}^{(n)}\right|^{2p}\left|b(X_{t_{k-1}}^{(n)})\right| \\
&\hspace{0.38cm}+C{\mathbb E}\left[\left|X_{t_{k}}^{(n)}-X_{t_{k-1}}^{(n)}\right|^{2}\left|X_{t_{k-1}}^{(n)}+\theta(X_{t_{k}}^{(n)}-X_{t_{k-1}}^{(n)})\right|^{2(p-1)}\,\Bigr|\,{\mathcal F}_{t_{k-1}}\right] \\
&\leq C\left(\frac{1}{n}+M_{2}^{(n)}(Z) \vee M_{2p}^{(n)}(Z)\right)\left(1+\left|X_{t_{k-1}}^{(n)}\right|^{2p}\right).
\end{align*}
Thus we obtain
\begin{align*}
{\mathbb E}\left[\left|X_{t_{\ell}}^{(n)}\right|^{2p}\,\Bigr|\,{\mathcal F}_{0}\right]
&=|x_{0}|^{2p}+\sum_{k=1}^{\ell}{\mathbb E}\left[\left|X_{t_{k}}^{(n)}\right|^{2p}-\left|X_{t_{k-1}}^{(n)}\right|^{2p}\,\Bigr|\,{\mathcal F}_{0}\right] \\
&\leq |x_{0}|^{2p}+C\left(1+n\left(M_{2}^{(n)}(Z) \vee M_{2p}^{(n)}(Z)\right)\right)
+C\left(\frac{1}{n}+M_{2}^{(n)}(Z) \vee M_{2p}^{(n)}(Z)\right)\sum_{k=1}^{\ell}{\mathbb E}\left[\left|X_{t_{k-1}}^{(n)}\right|^{2p}\,\Bigr|\,{\mathcal F}_{0}\right].
\end{align*}
Hence it follows by using  the discrete Gronwall inequality (cf. \cite{HJM}) that
\begin{align*}
{\mathbb E}\left[\left|X_{t_{\ell}}^{(n)}\right|^{2q}\,\Bigr|\,{\mathcal F}_{0}\right]
&\leq |x_{0}|^{2p}+C\left(1+n\left(M_{2}^{(n)}(Z) \vee M_{2p}^{(n)}(Z)\right)\right)+\left(|x_{0}|^{2p}+C\left(1+n\left(M_{2}^{(n)}(Z) \vee M_{2p}^{(n)}(Z)\right)\right)\right) \\
&\hspace{0.38cm}\times C\left(1+n\left(M_{2}^{(n)}(Z) \vee M_{2p}^{(n)}(Z)\right)\right)\exp\left\{C\left(1+n\left(M_{2}^{(n)}(Z) \vee M_{2p}^{(n)}(Z)\right)\right)\right\},
\end{align*}
which concludes the second statement.
\end{proof}

\begin{Lemma}
\label{lem:5.8}
Suppose that Assumptions \ref{ass:3.1} and \ref{ass:3.2} hold.
Then for any $k \in \{1,2,3\}$ and $i_{1},i_{2},\ldots,i_{k} \in \{1,2,\ldots,d\}$ and $\ell \in \{1,2,\ldots,n\}$, it holds that
$$
{\mathbb E}\left[\prod_{j=1}^{k}(X_{t_{\ell}}^{(n)}-X_{t_{\ell-1}}^{(n)})^{i_{j}}-\prod_{j=1}^{k}(\eta_{\ell}^{(n)}(t_{\ell})-X_{t_{\ell-1}}^{(n)})^{i_{j}}\,\Bigggr|\,{\mathcal F}_{t_{\ell-1}}\right]=0.
$$
\end{Lemma}
\begin{proof}
We set $(\Delta Z_{\ell}^{(n)})^{0}:=T/n$ and $(W_{t_{\ell}}-W_{t_{\ell-1}})^{0}:=T/n$ to simplify the argument.
By {\bf A2-1}-{\bf A2-3}, we obtain
\begin{align*}
&{\mathbb E}\left[\prod_{j=1}^{k}(X_{t_{\ell}}^{(n)}-X_{t_{\ell-1}}^{(n)})^{i_{j}}-\prod_{j=1}^{k}(\eta_{\ell}^{(n)}(t_{\ell})-X_{t_{\ell-1}}^{(n)})^{i_{j}}\,\Bigggr|\,{\mathcal F}_{t_{\ell-1}}\right] \\
&={\mathbb E}\left[\prod_{h=1}^{k}\sum_{j=0}^{d}\sigma_{j}^{i_{h}}(X_{t_{\ell-1}}^{(n)})(\Delta Z_{\ell}^{(n)})^{j_{h}}-\prod_{h=1}^{k}\sum_{j=0}^{d}\sigma_{j}^{i_{h}}(X_{t_{\ell-1}}^{(n)})(W_{t_{\ell}}-W_{t_{\ell-1}})^{j_{h}}\,\Bigggl|\,{\mathcal F}_{t_{\ell-1}}\right] \\
&={\mathbb E}\left[\sum_{j_{1},j_{2},\ldots,j_{k} \in \{0,1,\ldots,d\}}\prod_{h=1}^{k}\sigma_{j_{h}}^{i_{h}}(X_{t_{\ell-1}}^{(n)})(\Delta Z_{\ell}^{(n)})^{j_{h}}-\sum_{j_{1},j_{2},\ldots,j_{k} \in \{0,1,\ldots,d\}}\prod_{h=1}^{k}\sigma_{j_{h}}^{i_{h}}(X_{t_{\ell-1}}^{(n)})(W_{t_{\ell}}-W_{t_{\ell-1}})^{j_{h}}\,\Bigggl|\,{\mathcal F}_{t_{\ell-1}}\right] \\
&=\sum_{j_{1},j_{2},\ldots,j_{k} \in \{0,1,\ldots,d\}}\left(\prod_{h=1}^{k}\sigma_{j_{h}}^{i_{h}}(X_{t_{\ell-1}}^{(n)})\right)\left({\mathbb E}\left[\prod_{h=1}^{k}(\Delta Z_{\ell}^{(n)})^{j_{h}}\right]-{\mathbb E}\left[\prod_{h=1}^{k}(W_{t_{\ell}}-W_{t_{\ell-1}})^{j_{h}}\right]\right) \\
&=0.
\end{align*}
\end{proof}
Next, we prove Theorem \ref{theo:3.3} using the above lemmas.

\subsection{Proof of Theorem \ref{theo:3.3}}
\label{sec:5.3}
\begin{proof}
From Lemma \ref{lem:5.6} and Taylor's theorem with Theorem \ref{theo:5.1}, we obtain
\begin{align}
\label{eq:5.10}
&{\mathbb E}[f(X_{T}^{(n)})]-{\mathbb E}[f(X_{T})]
={\mathbb E}[u(T,X_{T}^{(n)})]-{\mathbb E}[u(0,x_{0})] \\ \notag
&={\mathbb E}\left[\sum_{\ell=1}^{n}\biggl(u(t_{\ell},X_{t_{\ell}}^{(n)})-u(t_{\ell-1},X_{t_{\ell-1}}^{(n)})\biggr)\right]
={\mathbb E}\left[\sum_{\ell=1}^{n}\biggl(u(t_{\ell},X_{t_{\ell}}^{(n)})-u(t_{\ell},X_{t_{\ell}}^{t_{\ell-1},X_{t_{\ell-1}}^{(n)}})\biggr)\right] \\ \notag
&={\mathbb E}\left[\sum_{\ell=1}^{n}\left(\biggl(u(t_{\ell},X_{t_{\ell}}^{(n)})-u(t_{\ell},X_{t_{\ell-1}}^{(n)})\biggr)-\biggl(u(t_{\ell},\eta_{\ell}^{(n)}(t_{\ell}))-u(t_{\ell},X_{t_{\ell-1}}^{(n)})\biggr)\right)\right] \\ \notag
&\hspace{0.38cm}+{\mathbb E}\left[\sum_{\ell=1}^{n}\left(\biggl(u(t_{\ell},\eta_{\ell}^{(n)}(t_{\ell}))-u(t_{\ell},X_{t_{\ell-1}}^{(n)})\biggr)-\biggl(u(t_{\ell},X_{t_{\ell}}^{t_{\ell-1},X_{t_{\ell-1}}^{(n)}})-u(t_{\ell},X_{t_{\ell-1}}^{(n)})\biggr)\right)\right] \\ \notag
&={\mathbb E}\left[\sum_{\ell=1}^{n}\left(\sum_{k=1}^{3}\frac{1}{k!}\sum_{i_{1},i_{2},\ldots,i_{k} \in \{1,2,\ldots,d\}}\left.\frac{\partial^{k}u(t_{\ell},y)}{\partial y^{i_{1}}\partial y^{i_{2}}\ldots\partial y^{i_{k}}}\right|_{y=X_{t_{\ell-1}}^{(n)}}\left(\prod_{j=1}^{k}(X_{t_{\ell}}^{(n)}-X_{t_{\ell-1}}^{(n)})^{i_{j}}-\prod_{j=1}^{k}(\eta_{\ell}^{(n)}(t_{\ell})-X_{t_{\ell-1}}^{(n)})^{i_{j}}\right)\right)\right] \\ \notag
&\hspace{0.38cm}+{\mathbb E}\left[\sum_{\ell=1}^{n}\left(\sum_{k=1}^{3}\frac{1}{k!}\sum_{i_{1},i_{2},\ldots,i_{k} \in \{1,2,\ldots,d\}}\left.\frac{\partial^{k}u(t_{\ell},y)}{\partial y^{i_{1}}\partial y^{i_{2}}\ldots\partial y^{i_{k}}}\right|_{y=X_{t_{\ell-1}}^{(n)}}\left(\prod_{j=1}^{k}(\eta_{\ell}^{(n)}(t_{\ell})-X_{t_{\ell-1}}^{(n)})^{i_{j}}-\prod_{j=1}^{k}(X_{t_{\ell}}^{t_{\ell-1},X_{t_{\ell-1}}^{(n)}}-X_{t_{\ell-1}}^{(n)})^{i_{j}}\right)\right)\right] \\ \notag
&\hspace{0.38cm}+{\mathbb E}\left[\sum_{\ell=1}^{n}\left(R_{\ell}^{(n)}(X_{t_{\ell}}^{(n)})-R_{\ell}^{(n)}(X_{t_{\ell}}^{t_{\ell-1},X_{t_{\ell-1}}^{(n)}})\right)\right] \\ \notag
&=:I_{1}+I_{2}+I_{3},
\end{align}
where the reminder terms have the following form: for $Y=X_{t_{\ell}}^{(n)}$ and $Y=X_{t_{\ell}}^{t_{\ell-1},X_{t_{\ell-1}}^{(n)}}$, there exist corresponding $d \times d$ diagonal matrices $\theta_{\ell}^{(i_{1},i_{2},i_{3},i_{4})}(Y)$ such that for any $i \in \{1,2,\ldots,d\}$, $(\theta_{\ell}^{(i_{1},i_{2},i_{3},i_{4})}(Y))_{i}^{i} \in (0,1)$ and
$$
R_{\ell}^{(n)}(Y)=\frac{1}{4!}\sum_{i_{1},i_{2},i_{3},i_{4} \in \{1,2,\ldots,d\}}\left.\frac{\partial^{4}u(t_{\ell},y)}{\partial y^{i_{1}}\partial y^{i_{2}}\partial y^{i_{3}}\partial y^{i_{4}}}\right|_{y=X_{t_{\ell-1}}^{(n)}+\theta_{\ell}^{(i_{1},i_{2},i_{3},i_{4})}(Y)(Y-X_{t_{\ell-1}}^{(n)})}\prod_{j=1}^{4}(Y-X_{t_{\ell-1}}^{(n)})^{i_{j}}.
$$
Here, note that the reminder term $R_{\ell}^{(n)}(\eta_{\ell}^{(n)}(t_{\ell}))$ does not appear since it is offset.

\begin{itemize}
\item First, we estimate $I_{1}$.
By {\bf A2-1} and Lemma \ref{lem:5.8}, we obtain
\begin{align}
\label{eq:5.11}
I_{1}&=\sum_{\ell=1}^{n}\Bigggl(\sum_{k=1}^{3}\frac{1}{k!}\sum_{i_{1},i_{2},\ldots,i_{k} \in \{1,2,\ldots,d\}}{\mathbb E}\Bigggl[\left.\frac{\partial^{k}u(t_{\ell},y)}{\partial y^{i_{1}}\partial y^{i_{2}}\ldots\partial y^{i_{k}}}\right|_{y=X_{t_{\ell-1}}^{(n)}} \\ \notag
&\hspace{4.4cm}\times {\mathbb E}\left[\prod_{j=1}^{k}(X_{t_{\ell}}^{(n)}-X_{t_{\ell-1}}^{(n)})^{i_{j}}-\prod_{j=1}^{k}(\eta_{\ell}^{(n)}(t_{\ell})-X_{t_{\ell-1}}^{(n)})^{i_{j}}\,\Bigggr|\,{\mathcal F}_{t_{\ell-1}}\right]\Bigggr]\Bigggr) \\ \notag
&=0.
\end{align}
\item Next, we estimate $I_{2}$.
By {\bf A2-1}, Theorem \ref{theo:5.1}, Lemma \ref{lem:5.5} and Lemma \ref{lem:5.7}, we have
\begin{align}
\label{eq:5.12}
\left|I_{2}\right|
&\leq \sum_{\ell=1}^{n}\Bigggl(\sum_{k=1}^{3}\frac{1}{k!}\sum_{i_{1},i_{2},\ldots,i_{k} \in \{1,2,\ldots,d\}}{\mathbb E}\Bigggl[\left|\left.\frac{\partial^{k}u(t_{\ell},y)}{\partial y^{i_{1}}\partial y^{i_{2}}\ldots\partial y^{i_{k}}}\right|_{y=X_{t_{\ell-1}}^{(n)}}\right| \\ \notag
&\hspace{4.4cm}\times \left|{\mathbb E}\left[\prod_{j=1}^{k}(\eta_{\ell}^{(n)}(t_{\ell})-X_{t_{\ell-1}}^{(n)})^{i_{j}}-\prod_{j=1}^{k}(X_{t_{\ell}}^{t_{\ell-1},X_{t_{\ell-1}}^{(n)}}-X_{t_{\ell-1}}^{(n)})^{i_{j}}\,\Bigggr|\,{\mathcal F}_{t_{\ell-1}}\right]\right|\Bigggr]\Bigggr) \\ \notag
&\leq \frac{C}{n}\left(1+{\mathbb E}\left[\left|X_{t_{\ell-1}}^{(n)}\right|^{2(4r+3)}\right]\right) \\ \notag
&\leq \frac{C}{n}\exp\left\{C\left(1+n\left(M_{2}^{(n)}(Z) \vee M_{2(4r+3)}^{(n)}(Z)\right)\right)\right\}.
\end{align}
\item Finally, we estimate $I_{3}$.
Fix $\ell \in \{1,2,\ldots,n\}$.
By using the Cauchy-Schwarz inequality, Theorem \ref{theo:5.1}, {\bf A2-1} and Lemma \ref{lem:5.7}, we obtain
\begin{align}
\label{eq:5.13}
&{\mathbb E}\left[\left.\left|R_{\ell}^{(n)}(X_{t_{\ell}}^{(n)})\right|\,\right|\,{\mathcal F}_{t_{\ell-1}}\right] \\ \notag
&\leq C\sum_{i_{1},i_{2},i_{3},i_{4} \in \{1,2,\ldots,d\}}{\mathbb E}\left[\left|\left.\frac{\partial^{4}u(t_{\ell},y)}{\partial y^{i_{1}}\partial y^{i_{2}}\partial y^{i_{3}}\partial y^{i_{4}}}\right|_{y=X_{t_{\ell-1}}^{(n)}+\theta_{\ell}^{(i_{1},i_{2},i_{3},i_{4})}(X_{t_{\ell}}^{(n)})(X_{t_{\ell}}^{(n)}-X_{t_{\ell-1}}^{(n)})}\right|^{2}\,\Bigggr|\,{\mathcal F}_{t_{\ell-1}}\right]^{1/2} \\ \notag
&\hspace{7.2cm}\times {\mathbb E}\left[\left|\prod_{j=1}^{4}(X_{t_{\ell}}^{(n)}-X_{t_{\ell-1}}^{(n)})^{i_{j}}\right|^{2}\,\Bigggr|\,{\mathcal F}_{t_{\ell-1}}\right]^{1/2} \\ \notag
&\leq C\left(1+\left|X_{t_{\ell-1}}^{(n)}\right|^{4r}+{\mathbb E}\left[\left|X_{t_{\ell}}^{(n)}-X_{t_{\ell-1}}^{(n)}\right|^{4r}\,\Bigr|\,{\mathcal F}_{t_{\ell-1}}\right]\right)^{1/2}{\mathbb E}\left[\left|X_{t_{\ell}}^{(n)}-X_{t_{\ell-1}}^{(n)}\right|^{8}\,\Bigr|\,{\mathcal F}_{t_{\ell-1}}\right]^{1/2} \\ \notag
&\leq C\left(1+M_{4r}^{(n)}(Z)^{1/2}\right)\left(M_{8}^{(n)}(Z)^{1/2}+\frac{1}{n^{4}}\right)\left(1+\left|X_{t_{\ell-1}}^{(n)}\right|^{2(r+2)}\right).
\end{align}
Similarly, by using the Cauchy-Schwarz inequality, Theorem \ref{theo:5.1}, {\bf A2-1} and Lemma \ref{lem:5.4}, we have
\begin{align}
\label{eq:5.14}
{\mathbb E}\biggl[\Bigl|R_{\ell}^{(n)}(X_{t_{\ell}}^{t_{\ell-1},X_{t_{\ell-1}}^{(n)}})\Bigr|\,\Bigr|\,{\mathcal F}_{t_{\ell-1}}\biggr]
&\leq \frac{C}{n^{2}}\left(1+\left|X_{t_{\ell-1}}^{(n)}\right|^{2(r+2)}\right).
\end{align}
Thus by \eqref{eq:5.13}, \eqref{eq:5.14} and Lemma \ref{lem:5.7}, we obtain
\begin{align}
\label{eq:5.15}
\left|I_{3}\right|&\leq \sum_{\ell=1}^{n}\left({\mathbb E}\left[\left|R_{\ell}^{(n)}(X_{t_{\ell}}^{(n)})\right|\right]+{\mathbb E}\biggl[\Bigl|R_{\ell}^{(n)}(X_{t_{\ell}}^{t_{\ell-1},X_{t_{\ell-1}}^{(n)}})\Bigr|\biggr]\right) \\ \notag
&\leq C\left(1+M_{4r}^{(n)}(Z)^{1/2}\right)\exp\left\{C\left(1+n\left(M_{2}^{(n)}(Z) \vee M_{2(r+2)}^{(n)}(Z)\right)\right)\right\}\left(nM_{8}^{(n)}(Z)^{1/2}+\frac{1}{n}\right).
\end{align}
\end{itemize}
Therefore, it follows by combining \eqref{eq:5.10}, \eqref{eq:5.11}, \eqref{eq:5.12} and \eqref{eq:5.15} that
$$
\left|{\mathbb E}[f(X_{T}^{(n)})]-{\mathbb E}[f(X_{T})]\right| \leq C_{r}^{(n)}(Z)\left(nM_{8}^{(n)}(Z)^{1/2}+\frac{1}{n}\right).
$$
Hence we have established (i).

Next, let us see (ii), where 
\begin{equation*}
    M_{2p}^{(n)}(Z) = \mathbb{E}\left[\left((\Delta W^1)^2+(\Delta W^2)^2+\cdots+ (\Delta W^d)^2\right)^p \right].
\end{equation*}
Here $ \Delta W^{i}:= W_{n/T}^{i}-W_{0}^{i}$, $i=1,2,\ldots,d$ for the $d$-dimensional Brownian motion $ W $. 
By the elementary inequality 
\begin{equation*}
    \left(x^1+x^2+ \cdots + x^d\right)^p
    \leq d^{p-1} \left((x^1)^p+(x^2)^p+ \cdots + (x^d)^p\right), \quad x^1,x^2, \ldots, x^d \geq 0, 
\end{equation*}
we obtain that
\begin{equation*}
    M_{2p}^{(n)}(Z) \leq d^{p-1} \cdot d \cdot \mathbb{E}\left[(\Delta W^1)^{2p}\right]
    = d^p (2p-1)!! \left(\frac{T}{n} \right)^p.  
\end{equation*}
By this inequality, and by the 
inequality $ n \geq (8r+5) Td $,
we see that 
\begin{equation}\label{M2}
    n\left(M_{2}^{(n)}(Z) \vee M_{2(r+2)}^{(n)}(Z) \vee M_{2(4r+3)}^{(n)}(Z)\right)
    \leq n \cdot d\frac{T}{n}= dT
\end{equation}
and 
\begin{equation}\label{M4}
    M_{4r}^{(n)}(Z) \leq 1,
\end{equation}
so that we have \eqref{boundGauss}
with the upper bound $ K_d \leq d \sqrt{7\cdot 5 \cdot 3} = d \sqrt{105} $. 
For the lower bound, let us observe that, by the multinomial theorem,
\begin{equation*}
\begin{split}
    M_{2p}^{(n)}(Z) &=  \sum_{\substack{p_1,p_2, \ldots, p_d \geq 0 \\ p_1+p_2+ \cdots + p_d = p}}
    \frac{p!}{p_1!p_2! \cdots p_d!}
   \mathbb{E} \left[ (\Delta W^{1})^{2p_1}(\Delta W^{2})^{2p_2} \cdots
    (\Delta W^d)^{2p_d}
    \right] \\
    & =p!\left(\frac{T}{n}\right)^{p}\sum_{\substack{p_1,p_2, \ldots, p_d \geq 0 \\ p_1+p_2+ \cdots + p_d = p}}
    \left(\frac{(2p_1-1)!!}{p_1!}\frac{(2p_2-1)!!}{p_2!} \cdots \frac{(2p_d-1)!!}{p_d!} \right). 
    \end{split}
\end{equation*}
Since $ (2a-1)!!/a! \geq 1 $ for any $ a \in \mathbb{N} \cup \{0\} $, 
we see that 
\begin{equation*}
\begin{split}
    M_{2p}^{(n)}(Z) &\geq  p!\left(\frac{T}{n}\right)^{p}\sum_{\substack{p_1,p_2, \ldots, p_d \geq 0  \\ p_1+p_2+ \cdots + p_d = p}} 1  = \frac{(p+d-1)!}{(d-1)!}\left(\frac{T}{n}\right)^{p}.
    \end{split}
\end{equation*}
Thus substituting $ p = 4 $, we get the lower bound of $ K_d $. 

Next, let us consider the case (iii) where 
\begin{equation*}\label{Haarmoment}
    \begin{split}
        M_{2p}^{(n)}(Z)=
       {\mathbb E}\left[ \left(\sum_{j=1}^d ((\Delta Z^{(n)})^j)^2 \right)^p\right]
        =\sum_{j=1}^d{\mathbb E}\left[((\Delta Z^{(n)})^j)^{2p} \right]
        = 
        d(2^{K-1})^{p-1}\left(\frac{T}{n}\right)^p.
    \end{split}
\end{equation*}
Then, with the assumption $ n \geq T2^{K-1} $,
we retrieve \eqref{M2}, \eqref{M4} and also obtain 
\eqref{Haarbound}. 
For (iv), the observation 
\begin{equation*}\label{Walshmoment}
    \begin{split}
        M_{2p}^{(n)}(Z)={\mathbb E}\left[ \left(\sum_{j=1}^d ((\Delta Z^{(n)})^j)^2 \right)^p\right]
        = {\mathbb E}\left[\left( \sum_{j=1}^d \frac{T}{n} \right)^p\right]
        = d^p \left(\frac{T}{n}\right)^p
    \end{split}
\end{equation*}
together with the assumption $ n \geq T d $
leads to \eqref{Walshbound}. 
\end{proof}

\section*{Conclusions}
\label{sec:6}
We have provided two Euler-Maruyama approximations,based on the Haar system and the Walsh system, in Section \ref{sec:2}.
In Sections \ref{sec:3} and \ref{sec:5}, we theoretically showed that they have the same weak order $1$ of convergence as the standard Euler-Maruyama approximation by the Gaussian system, but the bounding constant can be very different 
depending on the choice of the mimicking random variables. 
The difference implied by Theorem \ref{theo:3.3} (iii) and (iv) is observed to be very sharp
by the example given as Proposition \ref{prop:3.new}. 
Whether the difference implied by 
Theorem \ref{theo:3.3} (ii) is observable or not 
is left open for future studies.
In Section \ref{sec:4}, 
we have experimentally confirmed that our schemes are more efficient than the Gaussian scheme in respect of computation time. 
In summary, we recommend to use
the Walsh scheme for high-dimensional EM approximations.

\section*{Acknowledgments}
\label{sec:7}
The authors would like to thank Kenji Yasutomi (Ritsumeikan Univ.), Takanori Adachi (Tokyo Metropolitan Univ.) and Gregory Markowsky (Monash Univ.) for their helpful comments.
The forth author was supported by JSPS KAKENHI Grant Number 17J05514.


\begin{thebibliography}{99}
\setlength{\parskip}{0cm}
\setlength{\itemsep}{0cm}

\bibitem{An-Ko}
F. Antonelli and A. Kohatsu-Higa.
Rate of convergence of a particle method to the solution of the McKean--Vlasov equation.
{\it Ann. Appl. Probab.}, 
{\bf 12}(2), 423-476, (2002).

\bibitem{Crisan}
D. Crisan and E. McMurray.
Cubature on Wiener space for McKean--Vlasov SDEs with smooth scalar interaction.
{\it Ann. Appl. Probab.}, 
{\bf 29}(1), 130-177, (2019).

\bibitem{EWHJJA}
W. E, J. Han and A. Jentzen.
Deep Learning-Based Numerical Methods for High-Dimensional Parabolic Partial Differential Equations and Backward Stochastic Differential Equations.
{\it Commun. Math. Stat.},
{\bf 5}(4), 349-380, (2017).

\bibitem{Fine}
N. J. Fine.
On the Walsh functions.
{\it Trans. Amer. Math. Soc.},
{\bf 65}(3), 372-414, (1949).

\bibitem{Gray}
F. Gray.
Pulse code communication.
US Patent 2,632,058. 1953-03-17.

\bibitem{HJJAEW}
J. Han, A. Jentzen and W. E.
Solving high-dimensional partial differential equations using deep learning.
{\it Proc. Natl. Acad. Sci.},
{\bf 115}(34), 8505-8510, (2018).

\bibitem{Harase}
S. Harase.
Conversion of Mersenne Twister to double-precision floating-point numbers.
{\it Math. Comput. Simulation},
{\bf 161}, 76-83, (2019).

\bibitem{HJM}
J. M. Holte.
Discrete Gronwall lemma and applications. 
{\it MAA North Central Section Meeting at the University of North Dakota},
(2009).

\bibitem{HCPHWX}
C. Hur\'e, H. Pham, and X. Warin.
Deep backward schemes for high-dimensional nonlinear PDEs.
{\it Math. Comp.},
{\bf 89}, 1547-1579, (2020).

\bibitem{INWS}
N. Ikeda and S. Watanabe.
{\it  Stochastic differential equations and diffusion processes}.
2nd edn. North-Holland, Amsterdam,
(1981).

\bibitem{JAKEP1}
A. Jentzen and P. E. Kloeden.
Overcoming the order barrier in the numerical approximation of SPDEs with additive space-time noise. 
{\it Proc. R. Soc. A},
{\bf 465}, 649-667, (2009).

\bibitem{JAKEP2}
A. Jentzen and P. E. Kloeden.
The Numerical Approximation of Stochastic Partial Differential Equations. 
{\it Milan J. Math.},
{\bf 77}, 205-244, (2009).

\bibitem{KEPPE}
P. E. Kloeden and E. Platen.
{\it Numerical Solutions of Stochastic Differential Equations}.
Springer-Verlag, Berlin,
(1992).

\bibitem{Kohatsu-Ogawa}
A. Kohatsu-Higa and S. Ogawa.
Weak rate of convergence for an Euler scheme of nonlinear SDE’s.
{\it Monte Carlo Methods Appl.},
{\bf 3}(4), 327-345, (1997).

\bibitem{KR2}
R. Kruse.
{\it Strong and Weak Approximation of Semilinear Stochastic Evolution Equations}.
Springer-Verlag, Cham,
(2014).

\bibitem{LL}
J-M. Lasry and P-L. Lions.
Mean field games
{\it Japan. J. Math.}
{\bf 2}, 229-260 (2007)

\bibitem{LV}
T. Lyons and N. Victoir.
Cubature on Wiener Space.
{\it Proceedings: Mathematical, Physical and Engineering Sciences},
{\bf 460} (2041), 169-198, (2004).

\bibitem{MMNT}
M. Matsumoto and T. Nishimura.
Mersenne Twister: A 623-dimensionally Equidistributed Uniform Pseudorandom Number Generator.
{\it ACM Trans. Model. Comput. Simul}.
{\bf 8}(1), 3-30, (1998).

\bibitem{MNG}
G. N. Milstein.
{\it Numerical Integration of Stochastic Differential Equations}.
Kluwer Academic, Dordrecht,
(1995).

\bibitem{OgawaI}
S. Ogawa.
Monte Carlo simulation of nonlinear diffusion processes.
{\it Japan J. Indust. Appl. Math.},
{\bf 9}, 25-33,  (1992).

\bibitem{OgawaII}
S. Ogawa.
Monte Carlo simulation of nonlinear diffusion processes, II.
{\it Japan J. Indust. Appl. Math.},
{\bf 11}, 31-45, (1994).

\bibitem{OgawaIII}
S. Ogawa. 
Some problems in the simulation of nonlinear diffusion processes.
{\it Math. Comput. Simulation},
{\bf 38}(1-3), 217-223, (1995).

\bibitem{Taguchi}
H-L. Ngo and D. Taguchi.
Semi-implicit Euler-Maruyama approximation for non-colliding particle systems.
{\it Ann. Appl. Probab.},
{\bf 30}(2), 673-705, (2020).

\bibitem{Walsh}
J. L. Walsh.
A closed set of normal orthogonal functions.
{\it Amer. J. Math.},
{\bf 45}(1), 5-24, (1923).
\end{thebibliography}
\end{document}